\documentclass[11pt,letterpaper]{article}

\usepackage[margin=1in]{geometry}
\usepackage{amsmath}
\usepackage{amssymb}
\usepackage{amsthm}
\usepackage{multirow}
\usepackage{enumerate}
\usepackage{longtable}
\usepackage{caption2}

\allowdisplaybreaks[1]

\newtheorem{theorem}{Theorem}
\newtheorem{corollary}[theorem]{Corollary}
\newtheorem{lemma}[theorem]{Lemma}
\newtheorem{proposition}[theorem]{Proposition}
\newtheorem{definition}[theorem]{Definition}

\theoremstyle{definition}
\newtheorem{example}[theorem]{Example}
\newtheorem{remark}[theorem]{Remark}

\numberwithin{theorem}{section}
\numberwithin{equation}{section}
\numberwithin{table}{section}

\newcommand{\bb}{\mathbb}
\newcommand{\Z}{\bb{Z}}

\newcommand{\R}{\bb{R}}

\newcommand{\cC}{\mathcal{C}}

\newcommand{\cP}{\mathcal{P}}
\newcommand{\cQ}{\mathcal{Q}}

\newcommand{\de}{\delta}
\newcommand{\De}{\Delta}

\newcommand{\Ga}{\Gamma}

\newcommand{\La}{\Lambda}

\newcommand{\var}{\varphi}
\newcommand{\Sig}{\Sigma}
\newcommand{\sig}{\sigma}

\newcommand{\pr}{\prime}

\newcommand{\sm}{\setminus}
\newcommand{\es}{\emptyset}

\newcommand{\ol}{\overline}
\newcommand{\lan}{\langle}
\newcommand{\ran}{\rangle}

\newcommand{\wti}{\widetilde}
\newcommand{\wh}{\widehat}

\newcommand{\Aut}{\text{Aut}}

\newcommand{\GR}{\text{GR}}

\newcommand{\Sym}{\text{Sym}}

\newcommand{\wt}{\text{wt}}

\newcommand{\sqm}{\sqrt{-1}}

\newcommand{\covol}{\text{covol}}
\newcommand{\vol}{\text{vol}}

\long\def\symbolfootnote[#1]#2{\begingroup%
\def\thefootnote{\fnsymbol{footnote}}\footnote[#1]{#2}\endgroup}

\begin{document}

\title{Constructions of Primitive Formally Dual Pairs Having Subsets with Unequal Sizes}
\author{Shuxing Li \and Alexander Pott}
\date{}
\maketitle

\symbolfootnote[0]{
S.~Li and A.~Pott are with the Faculty of Mathematics, Otto von Guericke University Magdeburg, 39106 Magdeburg, Germany (e-mail: shuxing.li@ovgu.de, alexander.pott@ovgu.de).
}

\begin{abstract}
The concept of formal duality was proposed by Cohn, Kumar and Sch\"urmann, which reflects a remarkable symmetry among energy-minimizing periodic configurations. This formal duality was later on translated into a purely combinatorial property by Cohn, Kumar, Reiher and Sch\"urmann, where the corresponding combinatorial objects were called formally dual pairs. Motivated by this surprising application on the energy minimization problem, we focus on the algebraic constructions of primitive formally dual pairs. It is worthy noting that almost all known examples of primitive formally dual pairs satisfy that the two subsets have the same size. Indeed, prior to this work, there was only one known example derived form computer search, which had subsets with unequal sizes in $\Z_2 \times \Z_4^2$. Inspired by this example, we propose a lifting construction framework and a recursive construction framework, which generate new primitive formally dual pairs from known ones. As an application, for $m \ge 2$, we obtain $m+1$ pairwise inequivalent primitive formally dual pairs in $\Z_2 \times \Z_4^{2m}$, which have subsets with unequal sizes.


\end{abstract}

\section{Introduction}

Let $\cC$ be a particle configuration in the Euclidean space $\R^n$. Let $f: \R^n \rightarrow \R$ be a potential function, which is used to measure the energy possessed by $\cC$. The energy minimization problem aims to find configurations $\cC \subset \R^n$ with a fixed density, whose energy is minimal with respect to a potential function $f$. In physics, the energy minimization problem amounts to find the ground states in a given space, with respect to a prescribed density and potential function. This problem is of great interest and notoriously difficult in general \cite[Section I]{CKS}. For instance, the famous sphere packing problem can be viewed as an extremal case of the energy minimization problem \cite[p. 123]{CKRS}.

In 2009, Cohn, Kumar and Sch\"{u}rmann considered a weaker version of the energy minimization problem, where the configurations under consideration are restricted to so called periodic configurations \cite{CKS}. A periodic configuration is formed by a union of finitely many translations of a lattice. For instance, let $\La$ be a lattice in $\R^n$, then $\cP=\bigcup_{i=1}^N (v_i+\La)$ is a periodic configuration formed by $N$ translations of $\La$. The \emph{density} of $\cP$ is defined to be $\de(\cP)=N/\covol(\La)$, where $\covol(\La)=\vol(\R^n/\La)$ is the volume of a fundamental domain of $\La$. Given a potential function $f: \R^n \rightarrow \R$, define its Fourier transformation
$$
\wh{f}(y)=\int_{\R^n}f(x)e^{-2\pi i\lan x,y \ran}dx,
$$
where $\lan \cdot, \cdot \ran$ is the inner product in $\R^n$. The potential functions involved in \cite{CKS} belong to the class of Schwartz function, so that their Fourier transformations are well-defined. For a Schwartz function $f: \R^n \rightarrow \R$ and a periodic configuration $\cP=\bigcup_{j=1}^N(v_j+\La)$ associated with a lattice $\La \subset \R^n$, define the \emph{average pair sum} of $f$ over $\cP$ as
$$
\Sig_f(\cP)=\frac{1}{N}\sum_{j,\ell=1}^N\sum_{x \in \La}f(x+v_j-v_\ell),
$$
which is used to measure the energy possessed by the periodic configuration $\cC$ with respect to the potential function $f$.

Based on numerical experiments, Cohn et al. observed that each energy-minimizing periodic configuration obtained in their simulations possesses an unexpected symmetry called formal duality \cite[Section VI]{CKS}. More precisely, if $\cP$ is an energy-minimizing periodic configuration, then numerous experiments suggested that there exists a periodic configuration $\cQ$, so that for \emph{each} Schwartz function $f$, we have
\begin{equation}\label{def-formaldual}
\Sig_f(\cP)=\de(\cP)\Sig_{\wh{f}}(\cQ).
\end{equation}
If two periodic configurations $\cP$ and $\cQ$ satisfy \eqref{def-formaldual} for each Schwartz function $f$, then they are called formally dual to each other \cite[Definition 2.1]{CKRS}. This formal duality among periodic configurations revealed a deep symmetry which has not been well understood.

Remarkably, Cohn, Kumar, Reiher and Sch\"{u}rmann realized that formal duality among a pair of periodic configurations can be translated into a purely combinatorial property \cite[Theorem 2.8]{CKRS}. Indeed, they introduced the concept of formally dual pairs in finite abelian groups, which is a combinatorial counterpart of formal duality \cite[Definition 2.9]{CKRS}. Let $\La \subset \R^n$ be a lattice with a basis containing $n$ vectors. The dual lattice of $\La$ is defined as
$$
\La^*=\{x \in \R^n \mid \lan x, y \ran \in \Z, \forall y \in \La \},
$$
in which $\lan \cdot, \cdot \ran$ is the inner product in $\R^n$. Let $\cP=\bigcup_{j=1}^N(v_j+\La)$ and $\cQ=\bigcup_{j=1}^M(w_j+\Ga)$ be two periodic configurations. Define $\cP-\cP$ to be the subset $\{ x-y \mid x,y \in \cP \}$. Suppose $\cP-\cP \subset \Ga^*$ and $\cQ-\cQ \subset \La^*$. Then, as observed in \cite[p. 129]{CKRS}, the two quotient groups $\Ga^*/\La$ and $\La^*/\Ga$ satisfy that $\Ga^*/\La \cong \La^*/\Ga \cong G$, where $G$ is a finite abelian group. Moreover, the two sets $S=\{v_j \mid 1 \le j \le N\}$ and $T=\{w_j \mid 1 \le j \le M\}$ can be regarded as subsets of $G$, so that $S$ corresponds to $\cP$ and $T$ corresponds to $\cQ$. Cohn et al.'s key observation was that, $\cP$ and $\cQ$ are formally dual if and only if $S$ and $T$ form a formally dual pair in $G$ (see Definition~\ref{def-iso} for the concept of formally dual pairs).
Consequently, the formal duality among periodic configurations $\cP$ and $\cQ$ was reduced to the property of a pair of subsets $S$ and $T$ in a finite abelian group $G$.  Next, we give an illustrative example, describing how to derive a formally dual pair from a pair of formally dual periodic configurations.

\begin{example}
Let $n$ be a positive integer with $n \equiv 1 \bmod 4$. Let
$$
D_n=\{(x_1,x_2,\ldots,x_n) \in \Z^n \mid \sum_{i=1}^n x_i \equiv 0 \bmod2\}
$$
be the checkerboard lattice in $\R^n$. Set
$$
v_0=(0,0,\ldots,0), \quad v_1=(\frac12,\frac12,\ldots,\frac12), \quad v_2=(0,0,\ldots,0,1), \quad v_3=(\frac12,\frac12,\ldots,\frac12,-\frac12).
$$
Define a periodic configuration $D_n^+=D_n \cup (v_1+D_n)$ associated with the lattice $D_n$. Since $n \equiv 1 \pmod 4$, the two periodic configurations $D_{n}^+$ and $D_{n}^+$ are formally dual to each other \cite[Proposition 1]{CKS}.

Next, we derive a formally dual pair in $\Z_4$, corresponding to the pair of formally dual periodic configurations $D_{n}^+$ and $D_{n}^+$. Note that the dual lattice $D_n^*=\bigcup_{i=0}^3 (v_i+D_n)$ \cite[Chapter 4, Section 7.4]{CS}, we can easily verify that the quotient group $D_{n}^*/D_{n} \cong \Z_4=\lan g \ran$, where $v_1+D_n$ and $v_3+D_n$ have order $4$ in the group $D_{n}^*/D_{n}$. Now we can construct a subset $S$ of $\Z_4$ corresponding to the periodic configuration $D_{n}^+$ in the following way. Note that there exists a group isomorphism $\phi: D_{n}^*/D_{n} \rightarrow \Z_4$, such that $\phi(v_i+D_{n})=g^i$ for each $0 \le i \le 3$. We can identify $v_i$ with the element $v_i+D_{n}$ in the quotient group $D_{n}^*/D_{n}$. Recall that $v_0$ and $v_1$ are the two translations of $D_{n}$ forming $D_{n}^+$. Therefore, we obtain a subset $S=\{\phi(v_0+D_{n}), \phi(v_1+D_{n})\}=\{1,g\} \subset \Z_4$ corresponding to $D_{n}^+$.

Similarly, since the quotient group $D_{n}^*/D_{n} \cong \Z_4=\lan g \ran$, we can construct a subset $T$ of $\Z_4$ corresponding to the periodic configuration $D_{n}^+$ in the following way. We know that there exists a group isomorphism $\phi^{\pr}: D_{n}^*/D_{n} \rightarrow \Z_4$, such that $\phi^{\pr}(v_i+D_{n})=g^{3i}$ for each $0 \le i \le 3$. Recall that $v_0$ and $v_1$ are the two translations of $D_{n}$ forming $D_{n}^+$. Therefore, we obtain a subset $T=\{\phi^{\pr}(v_0+D_{n}), \phi^{\pr}(v_1+D_{n})\}=\{1,g^3\} \subset \Z_4$ corresponding to $D_{n}^+$.

The two subsets $S=\{1,g\}$ and $T=\{1,g^3\}$ form a formally dual pair in the group $\Z_4=\lan g \ran$ corresponding to the pair of formally dual periodic configurations $D_{n}^+$ and $D_{n}^+$. We remark that although the choices of the group isomorphisms $\phi$ and $\phi^{\pr}$ are not unique, different choices lead to equivalent formally dual pairs in the sense of Definition~\ref{def-equiv} below.
\end{example}

To sum up, formally dual pairs offer a fresh viewpoint towards the energy-minimizing periodic configurations, in which the combinatorial approaches come into play. Let $S=\{v_j \mid 1 \le j \le N\}$ and $T=\{w_j \mid 1 \le j \le M\}$ be a formally dual pair in a finite abelian group $G$. Then for each pair of lattices $\La$ and $\Ga$, satisfying $\Ga^*/\La \cong \La^*/\Ga \cong G$, we have that $\cP=\bigcup_{j=1}^N (v_j+\La)$ and $\cQ=\bigcup_{j=1}^M (w_j+\Ga)$ are formally dual periodic configurations. Hence, from a formally dual pair $S$ and $T$ in $G$, we derive $\cP$ and $\cQ$, which are two candidates of energy-minimizing periodic configurations. On the other hand, let $\La$ and $\Ga$ be two lattices such that $\Ga^*/\La \cong \La^*/\Ga \cong G$, where $G$ is a finite abelian group. Let $\cP$ be a periodic configuration associated with the lattice $\La$ and $\cQ$ be a periodic configuration associated with the lattice $\Ga$, such that $\cP-\cP \subset \Ga^*$ and $\cQ-\cQ \subset \La^*$. Then the nonexistence of formally dual pairs in $G$ implies that $\cP$ and $\cQ$ are not formally dual. Hence, the nonexistence of formally dual pairs in one finite abelian group $G$ rules out infinitely many potential pairs of formally dual periodic configurations and the arguments involved are purely combinatorial.

Below, we give a brief summary of known results about formally dual pairs. Some initial results were included in the pioneering work \cite{CKRS,CKS}. A main conjecture due to Cohn, Kumar, Reiher and Sch\"{u}rmann \cite[p. 135]{CKRS}, states that there are no primitive formally dual pairs in cyclic groups, except two small examples (see Definition~\ref{def-primi} for the concept of primitive formally dual pairs). This conjecture was proved for cyclic groups of prime power order, where Sch\"uler confirmed the odd prime power case \cite{Sch} and Xia confirmed the even prime power case \cite{Xia}. When the order of the cyclic group is a product of two prime powers, Malikiosis showed that the conjecture holds true in many cases \cite{Mali}. In \cite[Section 4.2]{LPS}, the authors proposed a new viewpoint towards the conjecture, by building a connection between the two known examples of primitive formally dual pairs in cyclic groups and cyclic relative difference sets. Besides, a systematic study of formally dual pairs in finite abelian groups was presented in \cite{LPS}, which contains constructions, classifications, nonexistence results and enumerations.

Let $S$ and $T$ be a primitive formally dual pair in $G$. Almost all known primitive formally dual pairs satisfy $|S|=|T|$. Indeed, there was only one known exception in \cite[Example 3.22]{LPS}, which gave a primitive formally dual pair $S$ and $T$ in $\Z_2 \times \Z_4^2$, with $|S|=4$ and $|T|=8$. This example motivates us to consider the algebraic construction of primitive formally dual pairs having subsets with unequal sizes. In fact, when $m \ge 2$, we construct $m+1$ pairwise inequivalent primitive formally dual pairs in $\Z_2 \times \Z_4^{2m}$. Our constructions are build upon a lifting construction framework and a recursive construction framework, which produce new primitive formally dual pairs from known ones.

The rest of the paper is organized as follows. In Section~\ref{sec2}, we give a brief introduction to formally dual pairs. A lifting construction framework is presented in Section~\ref{sec3}. Applying this lifting construction framework in Section~\ref{sec4}, we derive a direct construction of primitive formally dual pairs in $\Z_2 \times \Z_4^{2m}$, which leads to the first infinite family having subsets with unequal sizes. In Section~\ref{sec5}, we propose a recursive construction framework. Applying the recursive construction framework in Section~\ref{sec6}, we give the second infinite family in $\Z_2 \times \Z_4^{2m}$. Moreover, using the recursive construction framework, we can combine these two infinite families to generate new primitive formally dual pairs in $\Z_2 \times \Z_4^{2m}$. As a consequence, for $m \ge 2$, there are at least $m+1$ pairwise inequivalent primitive formally dual pairs in $\Z_2 \times \Z_4^{2m}$. Section~\ref{sec7} concludes the paper.

\section{Preliminaries}\label{sec2}

Throughout the paper, we always consider finite abelian groups $G$. Let $A_1$ and $A_2$ be two subsets of a group $G$. For each $y \in G$, define the \emph{weight enumerator} of $A_1$ and $A_2$ at $y$ as
$$
\nu_{A_1,A_2}(y)=|\{(a_1,a_2) \in A_1 \times A_2 \mid y=a_1a_2^{-1}\}|.
$$
When $A_1=A_2$, we simply write $\nu_{A_1,A_2}(y)$ as $\nu_{A_1}(y)$.

We use $\Z[G]$ to denote the group ring. For $A \in \Z[G]$ with  nonnegative coefficients, we use $\{A\}$ to denote the underlying subset of $G$ corresponding to the elements of $A$ with positive coefficients and $[A]$ the multiset corresponding to $A$. For $A \in \Z[G]$ and $g \in G$, we use $[A]_g$ to denote the coefficient of $g$ in $A$. Suppose $A=\sum_{g \in G} a_gg \in \Z[G]$, then $A^{(-1)}$ is defined to be $\sum_{g \in G} a_gg^{-1}$. Suppose $A=\sum_{g \in G} a_gg \in \Z[G]$ and $B=\sum_{g \in G} b_gg \in \Z[G]$, then the product $AB$ is defined to be $\sum_{g \in G} (\sum_{h \in G}a_{gh^{-1}}b_h) g$. A \emph{character} $\chi$ of $G$ is a group homomorphism from $G$ to the multiplicative group of the complex field $\mathbb{C}$. For a group $G$, we use $\wh{G}$ to denote its character group. There exists a group isomorphism $\De:G \rightarrow \wh{G}$, such that for each $y \in G$, we have $\chi_y:=\De(y) \in \wh{G}$. Therefore, $\wh{G}=\{\chi_y \mid y \in G\}$. A character $\chi \in \wh{G}$ is \emph{principal}, if $\chi(g)=1$ for each $g \in G$. A character $\chi \in \wh{G}$ is principal on a subgroup $H \leqslant G$, if $\chi(h)=1$ for each $h \in H$. For $\chi \in \wh{G}$ and $A=\sum_{g \in G} a_gg \in \Z[G]$, we use $\chi(A)$ to denote the character sum $\sum_{g \in G} a_g\chi(g)$. For a more detailed treatment of group rings and characters, please refer to \cite[Chapter 1]{Pott95}.

Now we are ready to introduce the definition of formally dual pairs.

\begin{definition}[Formally dual pair]\label{def-iso}
Let $\De$ be a group isomorphism from $G$ to $\wh{G}$, such that $\De(y)=\chi_y$ for each $y \in G$. Let $S$ and $T$ be subsets of $G$. Then $S$ and $T$ form a formally dual pair in $G$ under the isomorphism $\De$, if for each $y \in G$,
\begin{equation}\label{eqn-def}
|\chi_y(S)|^2=\frac{|S|^2}{|T|}\nu_T(y).
\end{equation}
\end{definition}

\begin{remark}\label{rem-def}
\quad
\begin{itemize}
\item[(1)] According to \cite[Remark 2.10]{CKRS}, the roles of the two subsets $S$ and $T$ in a formally dual pair are interchangeable, in the sense that \eqref{eqn-def} holds for each $y \in G$, if and only if
    \begin{equation}\label{eqn-def2}
      |\chi_y(T)|^2=\frac{|T|^2}{|S|}\nu_S(y)
    \end{equation}
    holds for each $y \in G$.
\item[(2)] By Definition~\ref{def-iso}, formal duality depends only on $SS^{(-1)}$ and $TT^{(-1)}$. For each $g_1, g_2\in G$, suppose that $S^{\pr}=\{g_1x \mid x \in S\}$ is a translation of $S$ and $T^{\pr}=\{g_2x \mid x \in T\}$ is a translation of $T$. Then $S^{\pr}$ and $T^{\pr}$ also form a formally dual pair in $G$. Hence, formal duality is invariant under translation.
\item[(3)] By \cite[Proposition 2.9]{LPS}, we know that $S$ and $T$ form a formally dual pair in $G$ under the isomorphism $\De_1$ if and only if $S$ and $\De_2^{-1}(\De_1(T))$ form a formally dual pair in $G$ under the isomorphism $\De_2$. Thus, Definition~\ref{def-iso} does not depend on the specific choice of $\De$. From now on, by referring to a formally dual pair, we always assume a proper group isomorphism is chosen. In our concrete constructions below, we always use a group isomorphism $\De: G \rightarrow \wh{G}$, such that $\De(y)=\chi_y$ for each $y \in G$. Therefore, once we specify how the character $\chi_y$ is defined, the group isomorphism $\De$ follows immediately.
\item[(4)] By \cite[Theorem 2.8]{CKRS}, we must have $|G|=|S|\cdot|T|$. Hence, a formally dual pair in a group of nonsquare order, must contain two subsets with unequal sizes.
\end{itemize}
\end{remark}

To exclude some trivial examples of formally dual pairs, the concept of primitive formally dual pair was proposed in \cite[p. 134]{CKRS}.

\begin{definition}[Primitive formally dual pair]\label{def-primi}
For a subset $S$ of a group $G$, define $S$ to be a primitive subset of $G$, if $S$ is not contained in a coset of a proper subgroup of $G$ and $S$ is not a union of cosets of a nontrivial subgroup in $G$. For a formally dual pair $S$ and $T$ in $G$, it is a primitive formally dual pair, if both $S$ and $T$ are primitive subsets.
\end{definition}

\begin{remark}\label{rem-primitive}
According to \cite[Remark 2.8(1)]{LPS}, given a formally dual pair $S$ and $T$ in $G$, the fact that neither of $S$ and $T$ is contained in a coset of a proper subgroup of $G$, guarantees that $S$ and $T$ form a primitive formally dual pair in $G$.
\end{remark}

A subset $S \subset G$ is called a (primitive) formally dual set in $G$, if there exists a subset $T \subset G$, such that $S$ and $T$ form a (primitive) formally dual pair in $G$. The following lemma presents a simple characterization of primitive subsets, which will be used later.

\begin{lemma}[{\cite[Lemma 2.18]{LPS}}]\label{lem-priset}
A set $S$ is contained in a coset of a proper subgroup $H$ of $G$ if and only if there exists a nonprincipal character $\chi$, such that $|\chi(S)|^2=|S|^2$.
\end{lemma}

The following definition concerns the equivalence of formally dual pairs \cite[Definition 2.17]{LPS}. Given a group $G$, we use $\Aut(G)$ to denote its automorphism group.

\begin{definition}[Equivalence of formally dual pair]\label{def-equiv}
Let $S$ and $S^\pr$ be two formally dual sets in $G$. They are equivalent if there exist $g \in G$ and $\phi \in \Aut(G)$, such that
$$
S^\pr=g\phi(S).
$$
Moreover, let $S,T$ and $S^{\pr},T^{\pr}$ be two formally dual pairs in $G$. They are equivalent if one of $S$ and $T$ is equivalent to one of $S^{\pr}$ and $T^{\pr}$.
\end{definition}

Let $S$ and $T$ be a formally dual pair in $G$. Suppose $S^\pr$ is equivalent to $S$. Then, by \cite[Proposition 2.16]{LPS}, there exists a subset $T^\pr$, which is equivalent to $T$, so that $S^\pr$ and $T^\pr$ form a formally dual pair in $G$. Hence, as mentioned in Definition~\ref{def-equiv}, the equivalence of formally dual pairs can be reduced to the equivalence of formally dual sets. For $A \in \Z[G]$, the multiset
$$
[[AA^{(-1)}]_g \mid g \in G]
$$
is called the \emph{difference spectrum} of $A$. The multiset
$$
[|\chi(A)|^2 \mid \chi \in \wh{G}]
$$
is called the \emph{character spectrum} of $A$. Clearly, both difference spectrum and character spectrum are invariants with respect to the equivalence of formally dual sets. Later on, we will use them to distinguish inequivalent primitive formally dual pairs.

Next, we mention a very powerful product construction.

\begin{proposition}[Product construction]\label{prop-prod}
Let $S_1$ and $T_1$ be a primitive formally dual pair in $G_1$. Let $S_2$ and $T_2$ be a primitive formally dual pair in $G_2$. Then $S_1 \times S_2$ and $T_1 \times T_2$ form a primitive formally dual pair in $G_1 \times G_2$.
\end{proposition}
\begin{proof}
By \cite[Lemma 3.1]{CKRS}, we know that $S_1 \times S_2$ and $T_1 \times T_2$ form a formally dual pair. Next, we show that $S_1 \times S_2$ and $T_1 \times T_2$ form a primitive formally dual pair in $G_1 \times G_2$. Otherwise, suppose that $S_1 \times S_2$ or $T_1 \times T_2$ is contained in a coset of a proper subgroup of $G_1 \times G_2$. Then, without loss of generality, we can assume that $S_1 \times S_2$ is a contained in a coset of a proper subgroup of $G_1 \times G_2$. By Lemma~\ref{lem-priset}, there exists a nonprincipal character $\chi \in \wh{G_1 \times G_2}$, such that $|\chi(S_1 \times S_2)|^2=|S_1 \times S_2|^2=|S_1|^2|S_2|^2$. For $i \in \{1,2\}$,  define $\chi_i$ to be the restriction of $\chi$ on $G_i$. Note that $|\chi(S_1 \times S_2)|^2=|\chi_1(S_1)|^2|\chi_2(S_2)|^2=|S_1|^2|S_2|^2$, which forces $|\chi_1(S_1)|^2=|S_1|^2$ and $|\chi_2(S_2)|^2=|S_2|^2$. Since $\chi$ is nonprincipal, then at least one of $\chi_1$ and $\chi_2$ is nonprincipal. Using Lemma~\ref{lem-priset} again, there exists $i \in \{1,2\}$ such that $S_i$ is not a primitive subset of $G_i$. This contradicts the fact that $S_i$ and $T_i$ form a primitive formally dual pair in $G_i$, where $i \in \{1,2\}$.
\end{proof}

Finally, we describe the well known Fourier inversion formula, which says a group ring element is uniquely determined by its character values.

\begin{proposition}[Fourier inversion formula] \label{prop-fourier}
Let $G$ be a group and let $A=\sum_{g \in G} a_gg \in \Z[G]$. Then for each $g \in G$, we have
$$
a_g=\frac{1}{|G|}\sum_{\chi \in \wh{G}} \chi(A)\ol{\chi(g)}.
$$
Consequently, for $A, B \in \Z[G]$, we have $A=B$ if and only if $\chi(A)=\chi(B)$ for each $\chi \in \wh{G}$.
\end{proposition}

\section{A lifting construction framework}\label{sec3}

In this section, we introduce a lifting construction framework, which generates new primitive formally dual pairs from known ones. Remarkably, this framework produces primitive formally dual pairs in which the two subsets have unequal sizes.

We first introduce some notation which will be used throughout the rest of this section. Let $G$ be a group of square order. Let $S$ and $T$ be a primitive formally dual pair in $G$ under the isomorphism $\De$, with $\De(y)=\chi_y$ for each $y \in G$. Suppose $|S|=|T|=\sqrt{|G|}$ and $S$ can be partitioned into two subsets $S_0$ and $S_1$. Let $T_0$ and $T_1$ be two subsets of $G$, such that $|T_0|+|T_1|=2|T|$. Define two subsets $S^\pr, T^\pr \subset \Z_2 \times G$ as follows:
\begin{align}
\begin{aligned}\label{eqn-lifting}
S^\pr&=\{(0,x) \mid x \in S_0\} \cup \{(1,x) \mid x \in S_1\}, \\
T^\pr&=\{(0,x) \mid x \in T_0\} \cup \{(1,x) \mid x \in T_1\}.
\end{aligned}
\end{align}
Clearly, $|S^\pr|=\sqrt{|G|}$ and $|T^\pr|=2\sqrt{|G|}$. For each $w \in \Z_2$, define the character $\var_w \in \wh{\Z_2}$ as $\var_w(a)=(-1)^{wa}$ for each $a \in \Z_2$. For each $(w,z) \in \Z_2 \times G$, define the character $\phi_{w,z} \in \wh{\Z_2 \times G}$ as $\phi_{w,z}((a,b))=\var_w(a)\chi_z(b)$ for each $(a,b) \in \Z_2 \times G$.

The above paragraph indicates a lifting construction framework: starting from a primitive formally dual pair $S$ and $T$ in $G$ with $|S|=|T|$, we aim to generate a new formally dual pair $S^\pr$ and $T^\pr$ in $\Z_2 \times G$ with $|S^\pr|\ne|T^\pr|$. The next theorem provides necessary and sufficient conditions ensuring that $S^\pr$ and $T^\pr$ form a formally dual pair in $\Z_2 \times G$.

\begin{theorem}\label{thm-lifting}
Let $S^\pr$ and $T^\pr$ be the subsets defined in \eqref{eqn-lifting}. Then $S^\pr$ and $T^\pr$ form a primitive formally dual pair in $\Z_2 \times G$, if and only if the following holds:
\begin{equation*}
|\chi_z(T_0+T_1)|^2=\frac{4|T|^2}{|S|}(\nu_{S_0}(z)+\nu_{S_1}(z)), \quad \mbox{for each $z \in G$}
\end{equation*}
and
\begin{equation*}
|\chi_z(T_0-T_1)|^2=\frac{4|T|^2}{|S|}(\nu_{S_0,S_1}(z)+\nu_{S_1,S_0}(z)), \quad \mbox{for each $z \in G$}.
\end{equation*}
\end{theorem}
\begin{proof}
By definition, $S^\pr$ and $T^\pr$ form a formally dual pair if and only if for each $(w,z) \in \Z_2 \times G$,
\begin{equation}\label{eqn-eqn1}
|\phi_{w,z}(T^\pr)|^2=\frac{|T^\pr|^2}{|S^\pr|}\nu_{S^\pr}((w,z)).
\end{equation}
Note that
\begin{align*}
S^{\pr}S^{\pr(-1)}&=\sum_{z \in [S_0S_0^{(-1)}+S_1S_1^{(-1)}]} (0,z) + \sum_{z \in [S_0S_1^{(-1)}+S_1S_0^{(-1)}]} (1,z), \\
T^{\pr}T^{\pr(-1)}&=\sum_{z \in [T_0T_0^{(-1)}+T_1T_1^{(-1)}]} (0,z) + \sum_{z \in [T_0T_1^{(-1)}+T_1T_0^{(-1)}]} (1,z).
\end{align*}
By splitting into the two cases $w=0$ and $w=1$, \eqref{eqn-eqn1} is equivalent to
\begin{align*}
|\chi_z(T_0+T_1)|^2&=\frac{4|T|^2}{|S|}(\nu_{S_0}(z)+\nu_{S_1}(z)), \quad \mbox{for each $z \in G$}, \\
|\chi_z(T_0-T_1)|^2&=\frac{4|T|^2}{|S|}(\nu_{S_0,S_1}(z)+\nu_{S_1,S_0}(z)), \quad \mbox{for each $z \in G$}.
\end{align*}
\end{proof}

\begin{remark}
By definition, $S^\pr$ and $T^\pr$ form a formally dual pair if and only if for each $(w,z) \in \Z_2 \times G$,
\begin{equation*}
|\phi_{w,z}(S^\pr)|^2=\frac{|S^\pr|^2}{|T^\pr|}\nu_{T^\pr}((w,z)).
\end{equation*}
In particular, for $w=0$, we have
\begin{align*}
|\phi_{0,z}(S^\pr)|^2&=|\chi_z(S)|^2=\frac{|S|^2}{|T|} \nu_T(z), \\
\frac{|S^\pr|^2}{|T^\pr|}\nu_{T^\pr}((0,z))&=\frac{|S|^2}{2|T|} (\nu_{T_0}(z)+\nu_{T_1}(z)).
\end{align*}
Thus, $S^\pr$ and $T^\pr$ form a formally dual pair only if $\nu_{T_0}(z)+\nu_{T_1}(z)=2\nu_T(z)$. Summing over the elements of $G$ on both sides, we have $|T_0|^2+|T_1|^2=2|T|^2$. Together with $|T_0|+|T_1|=2|T|$, we derive that $|T_0|=|T_1|=|T|$.
\end{remark}

In \eqref{eqn-lifting}, the two subsets $T_0$ and $T_1$ must be related to $T$ in certain way. Throughout the rest of this paper, we always consider the case that $T_0=T$ and $T_1=T^{(-1)}$. Hence, we define
\begin{align}
\begin{aligned}\label{eqn-lifting2}
S^{\pr\pr}&=\{(0,x) \mid x \in S_0\} \cup \{(1,x) \mid x \in S_1\}, \\
T^{\pr\pr}&=\{(0,x) \mid x \in T\} \cup \{(1,x) \mid x \in T^{(-1)}\}.
\end{aligned}
\end{align}
In this case, the necessary and sufficient conditions of Theorem~\ref{thm-lifting} can be further simplified. As a preparation, we need the following lemma which concerns the form of a subgroup of $\Z_2 \times G$.

\begin{lemma}\label{lem-subgp}
Let $H=(\{0\} \times H_0) \cup (\{1\} \times H_1)$ be a proper subgroup of $\Z_2 \times G$ with $H_0 \ne \es$ and $H_1 \ne \es$. Then $H=\Z_2 \times H_0$ and $H_0$ is a proper subgroup of $G$.
\end{lemma}
\begin{proof}
Since $H$ is a subgroup of $\Z_2 \times G$, then $H \cap (\{0\} \times G)=\{0\} \times H_0$ is a subgroup of $\{0\} \times G$, which implies that $H_0$ is a subgroup of $G$. Let $(1,h_1) \in H$. Since for every $(0,h_0) \in (\{0\} \times H_0)$, we have $(1,h_1)+(0,h_0) \in (\{1\} \times H_1)$, then $|H_0| \le |H_1|$. Similarly, for every $(1,h_1^{\pr}) \in (\{1\} \times H_1)$, we have $(1,h_1)+(1,h_1^{\pr}) \in (\{0\} \times H_0)$, which implies $|H_0| \ge |H_1|$. Hence, $|H_0|=|H_1|$ and $|H|=2|H_0|$. Since $H$ and $H_0$ are both subgroups, we have
\begin{equation}\label{eqn-eqn20}
\phi_{w,z}(H)=\begin{cases}
  |H| & \mbox{if $\phi_{w,z}$ is principal on $H$,} \\
  0   & \mbox{if $\phi_{w,z}$ is nonprincipal on $H$,}
\end{cases}
\end{equation}
and
\begin{equation}\label{eqn-eqn21}
\chi_z(H_0)=\begin{cases}
  |H_0| & \mbox{if $\chi_z$ is principal on $H_0$,} \\
  0   & \mbox{if $\chi_z$ is nonprincipal on $H_0$.}
\end{cases}
\end{equation}
Note that $\phi_{w,z}(H)=\chi_z(H_0)+(-1)^w\chi_z(H_1)$. Thus, \eqref{eqn-eqn20} and \eqref{eqn-eqn21} imply that for each $\chi_z \in \wh{G}$, either $\chi_z(H_0)=\chi_z(H_1)=|H_0|$ or $\chi_z(H_0)=\chi_z(H_1)=0$. By Proposition~\ref{prop-fourier}, we have $H_0=H_1$. Thus, $H=\Z_2 \times H_0$ and $H_0$ is a proper subgroup of $G$.
\end{proof}

Next, we give a necessary and sufficient condition for $S^{\pr\pr}$ and $T^{\pr\pr}$ being a primitive formally dual pair.

\begin{corollary}\label{cor-lifting}
Let $S^{\pr\pr}$ and $T^{\pr\pr}$ be the subsets defined in \eqref{eqn-lifting2}. Then $S^{\pr\pr}$ and $T^{\pr\pr}$ form a primitive formally dual pair in $\Z_2 \times G$ if and only if
\begin{equation*}
|\chi_z(T+T^{(-1)})|^2=\frac{4|T|^2}{|S|}(\nu_{S_0}(z)+\nu_{S_1}(z)), \quad \mbox{for each $z \in G$}.
\end{equation*}
\end{corollary}
\begin{proof}
By Theorem~\ref{thm-lifting}, we have that $S^{\pr\pr}$ and $T^{\pr\pr}$ form a formally dual pair if and only if
\begin{align}
\begin{aligned}\label{eqn-liftingfull}
|\chi_z(T+T^{(-1)})|^2&=\frac{4|T|^2}{|S|}(\nu_{S_0}(z)+\nu_{S_1}(z)), \quad \mbox{for each $z \in G$,} \\
|\chi_z(T-T^{(-1)})|^2&=\frac{4|T|^2}{|S|}(\nu_{S_0,S_1}(z)+\nu_{S_1,S_0}(z)), \quad \mbox{for each $z \in G$.}
\end{aligned}
\end{align}
For each $z \in G$, by summing the above two equations up, we get $|\chi_z(T)|^2=\frac{|T|^2}{|S|}\nu_S(z)$. Since $S$ and $T$ form a primitive formally dual pair in $G$, this equation always holds true. Thus, if one of the equations in \eqref{eqn-liftingfull} holds true, then so does the other. Next, we are going to show by contradiction that $S^{\pr\pr}$ is not contained in a coset of a proper subgroup of $\Z_2 \times G$. Suppose otherwise that $S^{\pr\pr}$ is contained in a coset of a proper subgroup of $\Z_2 \times G$. By Remark~\ref{rem-def}(2) and Lemma~\ref{lem-subgp}, applying a proper translation to $S^{\pr\pr}$, we can further assume that $S^{\pr\pr}=(\{0\} \times S_0) \cup (\{1\} \times S_1)$ is contained in a proper subgroup $\Z_2 \times H$ of $\Z_2 \times G$, where $H$ is a proper subgroup of $G$. Therefore, we know that $S=S_0 \cup S_1$ is contained in a proper subgroup $H$ of $G$, which contradicts the fact that $S$ is a primitive subset of $G$. Using a similar argument, we can show that $T^{\pr\pr}$ is not contained in a coset of a proper subgroup of $\Z_2 \times G$. By Remark~\ref{rem-primitive}, $S^{\pr\pr}$ and $T^{\pr\pr}$ form a primitive formally dual pair in $\Z_2 \times G$.
\end{proof}

\begin{remark}\label{rem-par}
\eqref{eqn-lifting2} presents a very general lifting construction framework to derive primitive formally dual pairs having subsets with unequal sizes. To apply this framework, we need to deal with the following two crucial points:
\begin{itemize}
\item[(1)] Choose a proper initial primitive formally dual pair $S$ and $T$ in a group $G$, satisfying $|S|=|T|$.
\item[(2)] Find a proper partition of $S$ into $S_0$ and $S_1$.
\end{itemize}
\end{remark}

In the next section, we will employ the lifting construction framework \eqref{eqn-lifting2} to produce the first infinite family of primitive formally dual pairs having two subsets with unequal sizes.

\section{A direct construction in $\Z_2 \times \Z_4^{2m}$}\label{sec4}

In this section, we give a direct construction of primitive formally dual pairs in $\Z_2 \times \Z_4^{2m}$, where the two subsets have unequal sizes.

First, we define the canonical characters on $\Z_4^n$ and $\Z_2 \times \Z_4^n$, which will be used later. For each $w \in \Z_2$, recall that the character $\var_w \in \wh{\Z_2}$ is defined as $\var_w(a)=(-1)^{wa}$ for each $a \in \Z_2$. For each $z=(z_1,z_2,\ldots,z_n) \in \Z_4^{n}$, define the character $\chi_z \in \wh{\Z_4^n}$ as $\chi_z(b)=(\sqrt{-1})^{z\cdot b}$ for each $b=(b_1,b_2,\ldots,b_n) \in \Z_4^{n}$, where $z\cdot b$ is defined as $\sum_{i=1}^n z_ib_i$. For each $(w,z) \in \Z_2 \times \Z_4^n$, define the character $\phi_{w,z} \in \wh{\Z_2 \times \Z_4^{n}}$ as $\phi_{w,z}((a,b))=\var_w(a)\chi_z(b)$ for each $(a,b) \in \Z_2 \times \Z_4^{n}$.

Now we introduce some notation which will be used throughout the rest of this paper. For $x=(x_1,x_2,\ldots,x_{n}) \in \Z_4^{n}$ and $j \in \Z_4$, define $\wt_j(x)=|\{ 1 \le i \le n \mid x_i=j \}|$. We write a multiset as $[A]=[a_i \lan z_i \ran \mid 1 \le i \le t]$, which means for each $1 \le i \le t$, the element $a_i$ occurs $z_i$ times in $[A]$. For two nonnegative integers $a$ and $b$, we use $\binom{a}{b}$ to denote the usual binomial coefficient, namely,
$$
\binom{a}{b}=\begin{cases}
\frac{\prod_{i=0}^{b-1} (a-i)}{b!} & \mbox{if $b \le a$,} \\
0, & \mbox{if $b>a$.}
\end{cases}
$$

%

Our direct construction is motivated by the following example described in \cite[Example 3.22]{LPS}.

\begin{example}\label{exam-TITO}
In the group $\Z_2\times\Z_4^2$, define two subsets
$$
S^\pr=\{(0,0,0), (0,0,1),(0,1,0), (1,1,1)\}
$$
and
$$
T^\pr=\{ (0,0,0), (0,0,1), (0,1,0), (0,1,1), (1,0,0), (1,0,3), (1,3,0), (1,3,3) \}.
$$
Then $S^\pr$ and $T^\pr$ form a primitive formally dual pair in $\Z_2\times\Z_4^2$. Define
$$
S=T=\{ (0,0), (0,1), (1,0), (1,1) \},
$$
then $S$ and $T$ form a primitive formally dual pair in $\Z_4^2$ (see \cite[Example 2.11, Proposition 3.2]{LPS}). Note that $S_0=\{(0,0),(0,1),(1,0)\}$ and $S_1=\{(1,1)\}$ form a partition of $S$. Therefore, this example fits into the lifting construction framework \eqref{eqn-lifting2}, and indeed, inspired us to propose the framework \eqref{eqn-lifting2}.
\end{example}

Next, we are going to show that Example~\ref{exam-TITO} is a member of an infinite family. In order to describe our construction, we need more notation. Define $J=\{0,1\} \subset \Z_4$. For $0 \le i \le 2m$, define a subset $B_{m,i}$ of $\Z_4^{2m}$ as
$$
B_{m,i}=\{ x \in \Z_4^{2m} \mid \wt_0(x)=2m-i, \wt_1(x)=i \}.
$$
From the viewpoint of the lifting construction framework \eqref{eqn-lifting2}, we identify the following pattern in Example~\ref{exam-TITO}:
\begin{itemize}
\item[(1)] $S=J \times J$ and $T=J \times J$ form the initial primitive formally dual pair in $\Z_4^2$.
\item[(2)] $S_0=B_{1,0} \cup B_{1,1}$ and $S_1=B_{1,2}$ form a partition of $S$.
\end{itemize}
By extending this pattern, we obtain the following direct construction.

\begin{theorem}\label{thm-dircon1}
Let $S=T=\prod_{j=1}^{2m} J$. Define
\begin{equation}\label{eqn-Sdircon1}
S_0=\sum_{\substack{0 \le i \le 2m \\ i \equiv 0,1 \bmod 4}}B_{m,i}, \quad S_1=\sum_{\substack{0 \le i \le 2m \\ i \equiv 2,3 \bmod 4}}B_{m,i}
\end{equation}
which form a partition of $S$. Let
\begin{equation}
\begin{aligned}\label{eqn-Tpdircon1}
S^{\pr}&=\{(0,x) \mid x \in S_0\} \cup \{(1,x) \mid x \in S_1\}, \\
T^{\pr}&=\{(0,x) \mid x \in T\} \cup \{(1,x) \mid x \in T^{(-1)}\}.
\end{aligned}
\end{equation}
Then $S^{\pr}$ and $T^{\pr}$ form a primitive formally dual pair in $\Z_2 \times \Z_4^{2m}$. Moreover,
\begin{align*}
&[[T^{\pr}T^{\pr(-1)}]_g \mid g \in \Z_2 \times \Z_4^{2m}] \\
=&[0\lan 2^{4m+1}-3^{2m+1}+2^{2m}\ran, 2\lan(m+1)2^{2m+1}\ran, 2^l\lan 2^{2m-l+1}\big(\binom{2m}{l-1}+\binom{2m}{l}\big)\ran \mid 2 \le l \le 2m+1].
\end{align*}
\end{theorem}

\begin{remark}\label{rem-dircon1}
\quad
\begin{itemize}
\item[(1)] In Theorem~\ref{thm-dircon1}, the subset $S$ is partitioned into $S_0$ and $S_1$, depending on the value of $\wt_1(x) \bmod4$, for each $x \in S$. 
\item[(2)] Suppose $S=T=\prod_{j=1}^{2m+1} J$. Let $S^\pr$ be an arbitrary subset of $\Z_2 \times \Z_4^{2m+1}$ and $T^\pr$ be the same as \eqref{eqn-Tpdircon1}. Then, $S^{\pr}$ and $T^{\pr}$ cannot be a primitive formally dual pair. Indeed, let $z=(1,1,\ldots,1) \in \Z_4^{2m+1}$. Then $|\phi_{0,z}(T^{\pr})|^2=2^{2m+2}$. By \eqref{eqn-def2}, we derive that $\nu_{S^{\pr}}((0,z))=\frac{1}{2}$, which is impossible. A similar argument in a group of the form $\Z_2 \times \Z_4^{2m}$, does not lead to such a contradiction.
\end{itemize}
\end{remark}


We know that $J$ and $J$ form a primitive formally dual pair in $\Z_4$ \cite[Section 3.1]{CKRS}. By Proposition~\ref{prop-prod}, $S=\prod_{j=1}^{2m} J$ and $T=\prod_{j=1}^{2m} J$ form a primitive formally dual pair in $\Z_4^{2m}$. Note that the construction in Theorem~\ref{thm-dircon1} fits into the lifting construction framework \eqref{eqn-lifting2}. By Corollary~\ref{cor-lifting}, in order to prove that $S^{\pr}$ and $T^{\pr}$ form a primitive formally dual pair, it suffices to show that
\begin{equation}\label{eqn-eqn2}
|\chi_z(T+T^{(-1)})|^2=\frac{4|T|^2}{|S|}(\nu_{S_0}(z)+\nu_{S_1}(z)), \quad \mbox{for each $z \in \Z_4^{2m}$}.
\end{equation}

Now we proceed to compute the left and right hand sides of \eqref{eqn-eqn2}. Firstly, we consider the right hand side. To understand $S_0S_0^{(-1)}$ and $S_1S_1^{(-1)}$, we need to compute $B_{m,i}B_{m,j}^{(-1)}$. For this purpose, more notation is needed. For $0 \le u,v \le 2m$ and $u+v \le 2m$, define
$$
C_{m,u,v}=\{ x \in \Z_4^{2m} \mid \wt_1(x)=u, \wt_3(x)=v, \wt_0(x)=2m-u-v \}.
$$
Hereafter, when we write $C_{m,u,v}$, we always assume that $0 \le u,v \le 2m$ and $u+v \le 2m$ hold. For $j \in \Z_4$, define
$$
K_{m,j}=\{x=(x_1,x_2,\ldots,x_{2m}) \in \Z_4^{2m} \mid \mbox{$x_i=j$ for some $1 \le i \le 2m$}\}.
$$
Then $\Z_4^{2m}$ can be partitioned as
$$
\Z_4^{2m}=(\bigcup_{\substack{0 \le u,v \le 2m \\ u+v \le 2m}} C_{m,u,v}) \bigcup K_{m,2}.
$$

We use $\Sym(n)$ to denote the symmetric group defined on $n$ elements. For $z=(z_1,z_2,\ldots,z_{2m}) \in \Z_4^{2m}$ and $\sig \in \Sym(2m)$, define $\sig(z)=(z_{\sig(1)},z_{\sig(2)},\ldots,z_{\sig(2m)})$. The action of $\sig$ on the elements of $\Z_4^{2m}$ can be naturally extended to the action on a subset of $\Z_4^{2m}$. For instance, we have
$$
\sig(B_{m,i})=\{\sig(x) \mid x \in B_{m,i}\}=B_{m,i}.
$$
Moreover, by the definition of $C_{m,u,v}$, for each $y \in C_{m,u,v}$, we have
$$
C_{m,u,v}=\{ \sig(y) \mid \sig \in \Sym(2m)\}.
$$

The following lemma concerns $C_{m,u,v}$, as well as the relation between $C_{m,u,v}$ and $B_{m,i}B_{m,j}^{(-1)}$.

\begin{lemma}\label{lem-CB}
\begin{itemize}
\item[(1)] $C_{m,u,v} \subset [B_{m,i}B_{m,j}^{(-1)}]$ if and only if $i=u+h$ and $j=v+h$ for some $0 \le h \le 2m-u-v$.
\item[(2)] For each $x \in C_{m,u,v}$, we have $[B_{m,u+h}B_{m,v+h}^{(-1)}]_x=\binom{2m-u-v}{h}$, where $0 \le h \le 2m-u-v$.
\end{itemize}
\end{lemma}
\begin{proof}
(1) Suppose $i=u+h$ and $j=v+h$ for some $0 \le h \le 2m-u-v$. Set
$$
y=(\underbrace{1,\ldots,1}_{u},\underbrace{0,\ldots,0}_{h},\underbrace{3,\ldots,3}_{v},\underbrace{0,\ldots,0}_{2m-u-v-h}) \in C_{m,u,v}.
$$
Then $y$ can be expressed as $y_1y_2^{-1}$, where
$$
y_1=(\underbrace{1,\ldots,1}_{u},\underbrace{1,\ldots,1}_{h},\underbrace{0,\ldots,0}_{v},\underbrace{0,\ldots,0}_{2m-u-v-h}) \in B_{m,u+h}
$$
and
$$
y_2=(\underbrace{0,\ldots,0}_{u},\underbrace{1,\ldots,1}_{h},\underbrace{1,\ldots,1}_{v}\underbrace{0,\ldots,0}_{2m-u-v-h}) \in B_{m,v+h}.
$$
Thus, $y \in [B_{u+h}B_{v+h}^{(-1)}]$. Recall that $C_{m,u,v}=\{ \sig(y) \mid \sig \in \Sym(2m)\}$. For each $\sig(y) \in C_{m,u,v}$, we can see that $\sig(y)=\sig(y_1)\sig(y_2)^{-1}$, where $\sig(y_1) \in B_{m,u+h}$ and $\sig(y_2) \in B_{m,v+h}$. Thus, $C_{m,u,v} \subset [B_{m,i}B_{m,j}^{(-1)}]$. Conversely, suppose $C_{m,u,v} \subset [B_{m,i}B_{m,j}^{(-1)}]$. Then there exist $z_1 \in B_{m,i}$ and $z_2 \in B_{m,j}$, such that $y=z_1z_2^{-1}$. Suppose there are exactly $h$ coordinates of $z_1$ and $z_2$ both with entry $1$. Then $y=z_1z_2^{-1} \in C_{m,u,v}$ implies that $i=u+h$ and $j=v+h$, where $0 \le h \le 2m-u-v$.

(2) Let $x$ and $y$ be two distinct elements of $C_{m,u,v}$. Since $C_{m,u,v}=\{ \sig(x) \mid \sig \in \Sym(2m)\}$, there exists $\sig_0 \in \Sym(2m)$, such that $y=\sig_0(x)$. For some $0 \le h \le 2m-u-v$, let $x_1 \in B_{m,u+h}$ and $x_2 \in B_{m,v+h}$, such that $x=x_1x_2^{-1}$. Then, we have $y=\sig_0(x)=\sig_0(x_1)\sig_0(x_2)^{-1}$, where $\sig_0(x_1) \in B_{m,u+h}$ and $\sig_0(x_2) \in B_{m,v+h}$. Consequently, we can see that $[B_{m,u+h}B_{m,v+h}^{(-1)}]_x=[B_{m,u+h}B_{m,v+h}^{(-1)}]_y$ for all $x,y \in C_{m,u,v}$. Without loss of generality, we can assume that
$$
x=(\underbrace{1,\ldots,1}_{u},\underbrace{3,\ldots,3}_{v},\underbrace{0,\ldots,0}_{2m-u-v}).
$$
This forces
$$
x_1=(\underbrace{1,\ldots,1}_{u},\underbrace{0,\ldots,0}_{v},\underbrace{\star,\ldots,\star}_{2m-u-v}),
$$
and
$$
x_2=(\underbrace{0,\ldots,0}_{u},\underbrace{1,\ldots,1}_{v},\underbrace{\star,\ldots,\star}_{2m-u-v}),
$$
where for each of the last $2m-u-v$ coordinates, the two entries in $x_1$ and $x_2$ are either both $0$ or both $1$, and exactly $h$ coordinates containing both $1$. Therefore, we conclude that $[B_{m,u+h}B_{m,v+h}^{(-1)}]_x=\binom{2m-u-v}{h}$ for each $x \in C_{m,u,v}$.
\end{proof}

Employing Lemma~\ref{lem-CB}, we can determine the multiset $[S_0S_0^{(-1)}+S_1S_1^{(-1)}]$.

\begin{proposition}\label{prop-S0S0invS1S1inv}
Let $S_0$ and $S_1$ be the subsets defined in \eqref{eqn-Sdircon1}. For $z \in \Z_4^{2m}$, we have
\begin{equation*}
  [S_0S_0^{(-1)}+S_1S_1^{(-1)}]_z=\begin{cases}
    0 & \mbox{if $z \in K_{m,2}$,} \\
    2^{2m-u-v} & \mbox{if $z \in C_{m,u,v}$, $u-v\equiv 0 \bmod 4$,} \\
    2^{2m-u-v-1} & \mbox{if $z \in C_{m,u,v}$, $u-v\equiv 1,3 \bmod 4$,} \\
    0 & \mbox{if $z \in C_{m,u,v}$, $u-v\equiv 2 \bmod 4$.}
  \end{cases}
\end{equation*}
\end{proposition}
\begin{proof}
By the definitions of $S_0$ and $S_1$, we have
\begin{equation}\label{eqn-S0S0inv}
S_0S_0^{(-1)}=\sum_{\substack{0 \le i,j \le 2m \\ i,j \equiv 0,1 \bmod4}} B_{m,i}B_{m,j}^{(-1)}, \quad S_1S_1^{(-1)}=\sum_{\substack{0 \le i,j \le 2m \\ i,j \equiv 2,3 \bmod4}} B_{m,i}B_{m,j}^{(-1)}
\end{equation}
Clearly, $[S_0S_0^{(-1)}+S_1S_1^{(-1)}]_z=0$ for each $z \in K_{m,2}$.

Next, we consider the case $z \in C_{m,u,v}$. Denote
$$
W=\{(0,0),(0,1),(1,0),(1,1),(2,2),(2,3),(3,2),(3,3)\} \subset (\{0,1,2,3\} \times \{0,1,2,3\}).
$$
For $z \in C_{m,u,v}$, by Lemma~\ref{lem-CB}(1) and \eqref{eqn-S0S0inv}, $z \in [S_0S_0^{(-1)}+S_1S_1^{(-1)}]$ if and only if there exists some $0 \le h \le 2m-u-v$, such that $((u+h) \bmod4, (v+h) \bmod4) \in W$.

If $u-v\equiv 0 \bmod4$, then for each $0 \le h \le 2m-u-v$, we have $((u+h) \bmod4, (v+h) \bmod4) \in W$. By Lemma~\ref{lem-CB}(2), $[S_0S_0^{(-1)}+S_1S_1^{(-1)}]_z=\sum_{h=0}^{2m-u-v} \binom{2m-u-v}{h}=2^{2m-u-v}$.

If $u-v \equiv 1 \bmod 4$, Table~\ref{tab-table1} lists all the possible triples $(u,v,h)$, such that $((u+h) \bmod4, (v+h) \bmod4) \in W$. Consequently, we have either
$$
[S_0S_0^{(-1)}+S_1S_1^{(-1)}]_z=\sum_{\substack{0 \le h \le 2m-u-v \\ h \equiv 0 \bmod2}} \binom{2m-u-v}{h}=2^{2m-u-v-1}
$$
or
$$
[S_0S_0^{(-1)}+S_1S_1^{(-1)}]_z=\sum_{\substack{0 \le h \le 2m-u-v \\ h \equiv 1 \bmod2}} \binom{2m-u-v}{h}=2^{2m-u-v-1}.
$$

\begin{table}[h]
\begin{center}
\caption{All $(u,v,h)$ triples satisfying $((u+h) \bmod4, (v+h) \bmod4) \in W$}\label{tab-table1}
\begin{tabular}{|c|c|} \hline
$(u,v)$                                     & $h$  \\ \hline
$u \equiv 1 \bmod4$, $v \equiv 0 \bmod4$  & $h \equiv 0,2 \bmod 4$  \\ \hline
$u \equiv 2 \bmod4$, $v \equiv 1 \bmod4$  & $h \equiv 1,3 \bmod 4$   \\ \hline
$u \equiv 3 \bmod4$, $v \equiv 2 \bmod4$  & $h \equiv 0,2 \bmod 4$   \\ \hline
$u \equiv 0 \bmod4$, $v \equiv 3 \bmod4$  & $h \equiv 1,3 \bmod 4$  \\ \hline
\end{tabular}
\end{center}
\end{table}

If $u-v\equiv 2 \bmod4$, then there exists no $0 \le h \le 2m-u-v$, such that $((u+h) \bmod4, (v+h) \bmod4) \in W$. Hence, $[S_0S_0^{(-1)}+S_1S_1^{(-1)}]_z=0$.

If $u-v \equiv 3 \bmod 4$, a similar argument as in the case of $u-v \equiv 1 \bmod 4$ gives
$$
[S_0S_0^{(-1)}+S_1S_1^{(-1)}]_z=2^{2m-u-v-1}.
$$
\end{proof}

Next, we compute the left hand side of \eqref{eqn-eqn2}.

\begin{proposition}\label{prop-TTinv}
Let $T=\prod_{j=1}^{2m} J$. For $z \in \Z_4^{2m}$, we have
$$
|\chi_z(T+T^{(-1)})|^2=\begin{cases}
  0 & \mbox{if $z \in K_{m,2}$,} \\
  2^{4m+2-u-v} & \mbox{if $z \in C_{m,u,v}$, $u-v \equiv 0 \bmod4$,} \\
  2^{4m+1-u-v} & \mbox{if $z \in C_{m,u,v}$, $u-v \equiv 1,3 \bmod4$,} \\
  0 & \mbox{if $z \in C_{m,u,v}$, $u-v \equiv 2 \bmod4$.}
\end{cases}
$$
\end{proposition}
\begin{proof}
For $y \in \Z_4$, it is easy to see that
$$
\chi_y(J)=\begin{cases}
  0 & \mbox{if $y=2$,} \\
  1+\sqm & \mbox{if $y=1$,} \\
  1-\sqm & \mbox{if $y=3$,} \\
  2 & \mbox{if $y=0$.}
\end{cases}
$$
Consequently, for $z=(z_1,z_2,\ldots,z_{2m}) \in \Z_4^{2m}$, we have
\begin{align*}
\chi_z(\prod_{j=1}^{2m} J)=\prod_{i=1}^{2m} \chi_{z_i}(J)&=\begin{cases}
  0 & \mbox{if $z \in K_{m,2}$,} \\
  (1+\sqm)^u(1-\sqm)^v2^{2m-u-v} & \mbox{if $z \in C_{m,u,v}$,}
\end{cases}\\
&=\begin{cases}
  0 & \mbox{if $z \in K_{m,2}$,} \\
  2^{2m-\frac{u+v}{2}}(\cos{\frac{(u-v)\pi}{4}}+\sqm\sin{\frac{(u-v)\pi}{4}}) & \mbox{if $z \in C_{m,u,v}$.}
\end{cases}
\end{align*}
Therefore, by the definition of $T$, we know that
$$
\chi_z(T+T^{(-1)})=\chi_z(\prod_{j=1}^{2m} J)+\ol{\chi_z(\prod_{j=1}^{2m} J)}=\begin{cases}
  0 & \mbox{if $z \in K_{m,2}$,} \\
  2^{2m+1-\frac{u+v}{2}}\cos{\frac{(u-v)\pi}{4}} & \mbox{if $z \in C_{m,u,v}$.}
\end{cases}
$$
Hence, we have
$$
|\chi_z(T+T^{(-1)})|^2=\begin{cases}
  0 & \mbox{if $z \in K_{m,2}$,} \\
  2^{4m+2-u-v}\cos^2{\frac{(u-v)\pi}{4}} & \mbox{if $z \in C_{m,u,v}$,}
\end{cases}
$$
which completes the proof.
\end{proof}

In the following, we proceed to compute the difference spectrum $[[T^{\pr}T^{\pr(-1)}]_g \mid g \in \Z_2 \times \Z_4^{2m}]$. For $0 \le u,v \le 2m$ and $u+v \le 2m$, define
$$
D_{m,u,v}=\{ x \in \Z_4^{2m} \mid \wt_1(x)=u, \wt_2(x)=v, \wt_0(x)=2m-u-v \}.
$$
Hereafter, when we write $D_{m,u,v}$, we always assume that $0 \le u,v \le 2m$ and $u+v \le 2m$ hold. By definition, $\Z_4^{2m}$ can be partitioned as
$$
\Z_4^{2m}=(\bigcup_{\substack{0 \le u,v \le 2m \\ u+v \le 2m}} D_{m,u,v}) \bigcup K_{m,3}=(\bigcup_{\substack{0 \le u,v \le 2m \\ u+v \le 2m}} D_{m,u,v}^{(-1)}) \bigcup K_{m,1}.
$$

The following is a preparatory lemma.

\begin{lemma}\label{lem-TTdiff}
Let $T=\prod_{j=1}^{2m} J$.
\begin{itemize}
\item[(1)]
For $x \in \Z_4^{2m}$, we have
$$
[TT^{(-1)}]_x=\begin{cases}
  0 & \mbox{if $x \in K_{m,2}$,} \\
  2^{l} & \mbox{if $x \in C_{m,u,2m-l-u}$, where $0 \le u \le 2m-l$ and $0 \le l \le 2m$.}
\end{cases}
$$
\item[(2)] For $x \in \Z_4^{2m}$, we have
$$
[TT]_x=\begin{cases}
  0 & \mbox{if $x \in K_{m,3}$,} \\
  2^{u} & \mbox{if $x \in D_{m,u,v}$, where $0 \le u \le 2m$.}
\end{cases}
$$
\item[(3)] For $x \in \Z_4^{2m}$, we have
$$
[T^{(-1)}T^{(-1)}]_x=\begin{cases}
  0 & \mbox{if $x \in K_{m,1}$,} \\
  2^{u} & \mbox{if $x \in D_{m,u,v}^{(-1)}$, where $0 \le u \le 2m$.}
\end{cases}
$$
\end{itemize}
\end{lemma}
\begin{proof}
We only prove (2), since the proof of (1) is similar and (3) follows from (2). Clearly, $[TT]_x=0$ for each $x \in K_{m,3}$. Let $x$ and $y$ be two distinct elements of $D_{m,u,v}$. For each $x \in D_{m,u,v}$, by the definition of $D_{m,u,v}$, we have $D_{m,u,v}=\{\sig(x) \mid \sig \in \Sym(2m)\}$. Therefore, there exists $\sig_0 \in \Sym(2m)$, such that $y=\sig_0(x)$. Suppose $x=x_1x_2$, where $x_1,x_2 \in T$. Then we have $y=\sig_0(x)=\sig_0(x_1)\sig_0(x_2)$, where $\sig_0(x_1), \sig_0(x_2) \in T$. Consequently, we can see that $[TT]_x=[TT]_y$ for all $x,y \in D_{m,u,v}$. Without loss of generality, we can assume that
$$
x=(\underbrace{1,\ldots,1}_{u},\underbrace{2,\ldots,2}_{v},\underbrace{0,\ldots,0}_{2m-u-v}).
$$
This forces
$$
x_1=(\underbrace{\star,\ldots,\star}_{u},\underbrace{1,\ldots,1}_{v},\underbrace{0,\ldots,0}_{2m-u-v}),
$$
and
$$
x_2=(\underbrace{\star,\ldots,\star}_{u},\underbrace{1,\ldots,1}_{v},\underbrace{0,\ldots,0}_{2m-u-v}),
$$
where for each of the first $u$ coordinates, the two entries in $x_1$ and $x_2$ are $0$ and $1$. Hence, we conclude that $[TT]_x=2^u$ for each $x \in D_{m,u,v}$.
\end{proof}

Now we can compute the multiset $[[T^{\pr}T^{\pr(-1)}]_g \mid g \in \Z_2 \times \Z_4^{2m}]$.

\begin{proposition}\label{prop-TpTpinv}
Let $T^\pr$ be the subset of $\Z_2 \times \Z_4^{2m}$ defined in \eqref{eqn-Tpdircon1}, then
\begin{align*}
&[[T^{\pr}T^{\pr(-1)}]_g \mid g \in \Z_2 \times \Z_4^{2m}] \\
=&[0\lan 2^{4m+1}-3^{2m+1}+2^{2m}\ran, 2\lan(m+1)2^{2m+1}\ran,2^l\lan 2^{2m-l+1}\big(\binom{2m}{l-1}+\binom{2m}{l}\big)\ran \mid 2 \le l \le 2m+1].
\end{align*}
\end{proposition}
\begin{proof}
Note that
\begin{equation}\label{eqn-eqn7}
T^{\pr}T^{\pr(-1)}=2\sum_{x \in [TT^{(-1)}]}(0,x)+\sum_{x \in [TT+T^{(-1)}T^{(-1)}]} (1,x),
\end{equation}
where $T=\prod_{j=1}^{2m} J$. It suffices to determine the two multisets $[[TT^{(-1)}]_x \mid x \in \Z_4^{2m}]$ and $[[TT+T^{(-1)}T^{(-1)}]_x \mid x \in \Z_4^{2m}]$.

Note that $|C_{m,u,v}|=\binom{2m}{u}\binom{2m-u}{v}$. By Lemma~\ref{lem-TTdiff}(1), for $0 \le l \le 2m$, we have
$$
|\{x \in \Z_4^{2m} \mid [TT^{-1}]_x=2^{l} \}|=\sum_{u=0}^{2m-l} |C_{m,u,2m-l-u}|=\sum_{u=0}^{2m-l} \binom{2m}{u}\binom{2m-u}{2m-l-u}=2^{2m-l}\binom{2m}{l}.
$$
Therefore, we have
\begin{equation}\label{eqn-eqn3}
[[TT^{(-1)}]_x \mid x \in \Z_4^{2m}]=[ 0\lan 4^{2m}-3^{2m}\ran, 2^l\lan 2^{2m-l}\binom{2m}{l} \ran \mid 0 \le l \le 2m ].
\end{equation}

According to Lemma~\ref{lem-TTdiff}(2)(3), we have
\begin{equation}\label{eqn-eqn4}
[TT]_x=\begin{cases}
  0 & \mbox{if $x \in K_{m,3}$,} \\
  2^u & \mbox{if $x \in \bigcup_{v=0}^{2m-u}D_{m,u,v}$,}
\end{cases}
\end{equation}
and
\begin{equation}\label{eqn-eqn5}
[T^{(-1)}T^{(-1)}]_x=\begin{cases}
  0 & \mbox{if $x \in K_{m,1}$,} \\
  2^u & \mbox{if $x \in \bigcup_{v=0}^{2m-u}D_{m,u,v}^{(-1)}$.}
\end{cases}
\end{equation}
By definition, $D_{m,0,v}=D_{m,0,v}^{(-1)}$ for each $0 \le v \le 2m$ and $D_{m,u,v} \cap D_{m,u,v}^{(-1)}=\es$ for each $1 \le u \le 2m$ and $0 \le v \le 2m-u$. Together with \eqref{eqn-eqn4} and \eqref{eqn-eqn5}, we have
\begin{equation*}
[TT+T^{(-1)}T^{(-1)}]_x=\begin{cases}
  0 & \mbox{if $x \in K_{m,1} \cap K_{m,3}$,} \\
  2 & \mbox{if $x \in (\bigcup_{v=0}^{2m} D_{m,0,v}) \bigcup (\bigcup_{v=0}^{2m-1} (D_{m,1,v} \cup D_{m,1,v}^{(-1)})$),} \\
  2^u & \mbox{if $x \in \bigcup_{v=0}^{2m-u} (D_{m,u,v} \cup D_{m,u,v}^{(-1)})$, $2 \le u \le 2m$.}
\end{cases}
\end{equation*}
Note that $|K_{m,1} \cap K_{m,3}|=4^{2m}-2\cdot3^{2m}+2^{2m}$ and $|D_{m,u,v}|=\binom{2m}{u}\binom{2m-u}{v}$. A direct computation shows
\begin{align}
\begin{aligned}\label{eqn-eqn6}
&[[TT+T^{(-1)}T^{(-1)}]_x \mid x \in \Z_4^{2m}] \\
=&[0\lan 4^{2m}-2\cdot3^{2m}+2^{2m} \ran, 2\lan(2m+1)2^{2m}\ran, 2^l\lan 2^{2m-l+1}\binom{2m}{l}\ran \mid 2 \le l \le 2m].
\end{aligned}
\end{align}
Combining \eqref{eqn-eqn7}, \eqref{eqn-eqn3} and \eqref{eqn-eqn6}, we complete the proof.
\end{proof}

Now we are ready to prove Theorem~\ref{thm-dircon1}.

\begin{proof}[Proof of Theorem~\ref{thm-dircon1}]
Applying Corollary~\ref{cor-lifting} and Propositions~\ref{prop-S0S0invS1S1inv} and~\ref{prop-TTinv}, we derive that $S^{\pr}$ and $T^{\pr}$ form a primitive formally dual pair in $\Z_2 \times \Z_4^{2m}$. The difference spectrum $[[T^{\pr}T^{\pr(-1)}]_g \mid g \in \Z_2 \times \Z_4^{2m}]$ follows from Proposition~\ref{prop-TpTpinv}.
\end{proof}

\section{A recursive construction framework}\label{sec5}

In this section, we propose a recursive construction framework. Roughly speaking, for $i \in \{1,2\}$, assume that $S_i$ and $T_i$ form a primitive formally dual pair in $\Z_2 \times G_i$, which is derived from the lifting construction framework \eqref{eqn-lifting2}. We find a method to combine the two primitive formally dual pairs $S_1,T_1$ and $S_2, T_2$, which leads to a new primitive formally dual pair in $\Z_2 \times G_1 \times G_2$. Thus, this method can be viewed as a recursive construction framework.

For a subset $A$ of a group $G$, we use $\theta(A,G)$ to denote the frequency of $0$ in the difference spectrum of $A$, i.e., in the multiset $[[AA^{(-1)}]_g \mid g \in G]$.

\begin{theorem}\label{thm-recur}
For $i \in \{1,2\}$, let $S_i$ and $T_i$ be a primitive formally dual pair in $G_i$. Partition $S_i$ as $S_i=S_{i0} \cup S_{i1}$. Define
\begin{align*}
S_i^{\pr}&=\{(0,x) \mid x \in S_{i0}\} \cup \{(1,x) \mid x \in S_{i1}\}, \\
T_i^{\pr}&=\{(0,x) \mid x \in T_i\} \cup \{(1,x) \mid x \in T_i^{(-1)}\}.
\end{align*}
For $i \in \{1,2\}$, assume that $S_i^\pr$ and $T_i^\pr$ form a primitive formally dual pair in $\Z_2 \times G_i$. Define two subsets of $\Z_2 \times G_1 \times G_2$ as
\begin{align}
\begin{aligned}\label{eqn-recur}
S^{\pr\pr}&=\{(0,x_1,x_2) \mid (x_1,x_2) \in S_0^{\pr\pr}\} \cup \{(1,x_1,x_2) \mid (x_1,x_2) \in S_1^{\pr\pr}\}, \\
T^{\pr\pr}&=\{(0,x_1,x_2) \mid (x_1,x_2) \in T_1 \times T_2\} \cup \{(1,x_1,x_2) \mid (x_1,x_2) \in T_1^{(-1)} \times T_2^{(-1)}\},
\end{aligned}
\end{align}
where
\begin{align*}
S_0^{\pr\pr}&=(S_{10} \times S_{20}) \cup (S_{11} \times S_{21}), \\
S_1^{\pr\pr}&=(S_{10} \times S_{21}) \cup (S_{11} \times S_{20}).
\end{align*}
For $i \in \{1,2\}$, let $\{\chi_{i,z_i} \mid z_i \in G_i\}$ be the set of all characters on $G_i$. Then $S^{\pr\pr}$ and $T^{\pr\pr}$ form a primitive formally dual pair in $\Z_2 \times G_1 \times G_2$ if and only if one of the following holds:
\begin{itemize}
\item[(1)] For each $\chi_{1,z_1} \in \wh{G_1}$, we have $\chi_{1,z_1}^2(T_1) \in \R$.
\item[(2)] For each $\chi_{2,z_2} \in \wh{G_2}$, we have $\chi_{2,z_2}^2(T_2) \in \R$.
\end{itemize}
Moreover, we have
$$
\theta(T^{\pr\pr},\Z_2 \times G_1 \times G_2)=2\cdot |G_1|\cdot |G_2|-|\{T_1T_1^{(-1)}\}|\cdot |\{T_2T_2^{(-1)}\}|-|\{T_1T_1+T_1^{(-1)}T_1^{(-1)}\}|\cdot |\{T_2T_2\}|.
$$

\end{theorem}
\begin{proof}
By Proposition~\ref{prop-prod}, $S_1 \times S_2$ and $T_1 \times T_2$ form a primitive formally dual pair in $G_1 \times G_2$. Moreover, since $S_{i0}$ and $S_{i1}$ form a partition of $S_i$, then by definition, $S_0^{\pr\pr}$ and $S_1^{\pr\pr}$ form a partition of $S_1 \times S_2$. Thus, the construction in \eqref{eqn-recur} fits into the lifting construction framework \eqref{eqn-lifting2}. We use $\psi_{z_1,z_2}$ to denote a character on $G_1 \times G_2$, such that for $g_i \in G_i$, we have $\psi_{z_1,z_2}(g_1,g_2)=\chi_{1,z_1}(g_1)\chi_{2,z_2}(g_2)$. By Corollary~\ref{cor-lifting}, $S^{\pr\pr}$ and $T^{\pr\pr}$ form a primitive formally dual pair in $\Z_2 \times G_1 \times G_2$ if and only if
\begin{equation}\label{eqn-eqn16}
|\psi_{z_1,z_2}(T_1 \times T_2+T_1^{(-1)}\times T_2^{(-1)})|^2=\frac{4|T_1|^2|T_2|^2}{|S_1||S_2|}(\nu_{S_0^{\pr\pr}}(z)+\nu_{S_1^{\pr\pr}}(z)), \quad \mbox{for each $z=(z_1,z_2) \in G_1 \times G_2$.}
\end{equation}
In the following, we denote the right hand size of \eqref{eqn-eqn16} as $RS$ and the left hand size as $LS$.  A direct computation shows that for $z=(z_1,z_2) \in G_1 \times G_2$,
\begin{align*}
\nu_{S_0^{\pr\pr}}(z)&=\nu_{S_{10}}(z_1)\nu_{S_{20}}(z_2)+\nu_{S_{11}}(z_1)\nu_{S_{21}}(z_2)+\nu_{S_{10},S_{11}}(z_1)\nu_{S_{20},S_{21}}(z_2)+\nu_{S_{11},S_{10}}(z_1)\nu_{S_{21},S_{20}}(z_2),\\
\nu_{S_1^{\pr\pr}}(z)&=\nu_{S_{10}}(z_1)\nu_{S_{21}}(z_2)+\nu_{S_{11}}(z_1)\nu_{S_{20}}(z_2)+\nu_{S_{10},S_{11}}(z_1)\nu_{S_{21},S_{20}}(z_2)+\nu_{S_{11},S_{10}}(z_1)\nu_{S_{20},S_{21}}(z_2).
\end{align*}
Consequently,
\begin{align*}
RS=\frac{4|T_1|^2|T_2|^2}{|S_1||S_2|}\big((&\nu_{S_{10}}(z_1)+\nu_{S_{11}}(z_1))(\nu_{S_{20}}(z_2)+\nu_{S_{21}}(z_2))\\
                                           &+(\nu_{S_{10},S_{11}}(z_1)+\nu_{S_{11},S_{10}}(z_1))(\nu_{S_{20},S_{21}}(z_2)+\nu_{S_{21},S_{20}}(z_2))\big).
\end{align*}
Since for $i \in \{1,2\}$, the sets $S_i^\pr$ and $T_i^\pr$ form a primitive formally dual pair in $\Z_2 \times G_i$, we have by \eqref{eqn-liftingfull} that
$$
RS=\frac{1}{4}\big( |\chi_{1,z_1}(T_1+T_1^{(-1)})|^2|\chi_{2,z_2}(T_2+T_2^{(-1)})|^2+|\chi_{1,z_1}(T_1-T_1^{(-1)})|^2|\chi_{2,z_2}(T_2-T_2^{(-1)})|^2 \big).
$$
A direct computation shows that
\begin{align*}
RS=\frac{1}{2}\big(&\chi_{1,z_1}^2(T_1)\chi_{2,z_2}^2(T_2)+\ol{\chi_{1,z_1}^2(T_1)}\ol{\chi_{2,z_2}^2(T_2)}+\chi_{1,z_1}^2(T_1)\ol{\chi_{2,z_2}^2(T_2)}+\ol{\chi_{1,z_1}^2(T_1)}\chi_{2,z_2}^2(T_2) \big)\\
                   &+2|\chi_{1,z_1}(T_1)|^2|\chi_{2,z_2}(T_2)|^2.
\end{align*}
Meanwhile, we have
\begin{align*}
LS&=|\chi_{1,z_1}(T_1)\chi_{2,z_2}(T_2)+\ol{\chi_{1,z_1}(T_1)}\ol{\chi_{2,z_2}(T_2)}|^2 \\
  &=\chi_{1,z_1}^2(T_1)\chi_{2,z_2}^2(T_2)+\ol{\chi_{1,z_1}^2(T_1)}\ol{\chi_{2,z_2}^2(T_2)}+2|\chi_{1,z_1}(T_1)|^2|\chi_{2,z_2}(T_2)|^2.
\end{align*}
Comparing $LS$ and $RS$, we can see that \eqref{eqn-eqn16} holds if and only if
$$
(\chi_{1,z_1}^2(T_1)-\ol{\chi_{1,z_1}^2(T_1)})(\chi_{2,z_2}^2(T_2)-\ol{\chi_{2,z_2}^2(T_2)})=0, \mbox{for each $(z_1,z_2) \in G_1 \times G_2$}.
$$
This amounts to that $\chi_{1,z_1}^2(T_1) \in \R$ for each $\chi_{1,z_1} \in \wh{G_1}$, or $\chi_{2,z_2}^2(T_2) \in \R$ for each $\chi_{2,z_2} \in \wh{G_2}$.

Finally, note that
\begin{align*}
T^{\pr\pr}T^{\pr\pr(-1)}=&2\sum_{(x_1,x_2)\in [(T_1\times T_2)(T_1\times T_2)^{(-1)}]}(0,x_1,x_2)\\
                   &+\sum_{(x_1,x_2)\in [(T_1\times T_2)(T_1\times T_2)+(T_1\times T_2)^{(-1)}(T_1\times T_2)^{(-1)}]}(1,x_1,x_2) \\
=&2\sum_{(x_1,x_2)\in [T_1T_1^{(-1)}\times T_2T_2^{(-1)}]}(0,x_1,x_2)+\sum_{(x_1,x_2)\in [(T_1T_1+T_1^{(-1)}T_1^{(-1)})\times T_2T_2]}(1,x_1,x_2).
\end{align*}
The equation of $\theta(T^{\pr\pr},\Z_2 \times G_1 \times G_2)$ follows immediately.
\end{proof}

\section{Inequivalent primitive formally dual pairs in $\Z_2 \times \Z_4^{2m}$}\label{sec6}

In this section, we will employ the recursive construction framework \eqref{eqn-recur} to generate the second infinite family of primitive formally dual pairs in $\Z_2 \times \Z_4^{2m}$. Moreover, using the recursive construction framework \eqref{eqn-recur}, we can combine the first infinite family in Theorem~\ref{thm-dircon1} and the second one, which leads to more inequivalent formally dual pairs in $\Z_2 \times \Z_4^{2m}$.

The second infinite family is motivated by the following example.

\begin{example}\label{exam-Teich}
In the group $\Z_2\times\Z_4^2$, define two subsets
$$
S^\pr=\{(0,0,0), (0,0,1),(0,1,0), (1,3,3)\}
$$
and
$$
T^\pr=\{ (0,0,0), (0,0,1), (0,1,0), (0,3,3), (1,0,0), (1,0,3), (1,3,0), (1,1,1) \}.
$$
Then $S^\pr$ and $T^\pr$ form a primitive formally dual pair in $\Z_2\times\Z_4^2$. It can be easily verified that $\theta(T^{\pr},\Z_2\times\Z_4^2)=9$. Define
$$
S=T=\{ (0,0), (0,1), (1,0), (3,3) \},
$$
then $S$ and $T$ form a primitive formally dual pair in $\Z_4^2$ (see \cite[Theorem 3.7]{LPS}). Therefore, this example fits into the lifting construction framework \eqref{eqn-lifting2}. We remark that this example is equivalent to Example~\ref{exam-TITO}.
\end{example}

Next, we are going to show that Example~\ref{exam-Teich} is a member of an infinite family. In order to describe our construction, we need more notation. Define
\begin{align*}
L&=\{(0,0),(0,1),(1,0),(3,3)\} \subset \Z_4^2 \\
L_1&=\{(0,0),(0,1),(1,0)\} \subset \Z_4^2 \\
L_2&=\{(3,3)\} \subset \Z_4^2
\end{align*}
where $L_1$ and $L_2$ form a partition of $L$. For $0 \le i \le m$, define a subset $E_{m,i}$ of $\Z_4^{2m}$ as
$$
E_{m,i}=\sum_{\substack{|\{1 \le j \le m \mid N_j=L_1\}|=i \\ |\{1 \le j \le m \mid N_j=L_2\}|=m-i}} \prod_{j=1}^m N_j.
$$
From the viewpoint of the lifting construction framework \eqref{eqn-lifting2}, we identify the following pattern in Example~\ref{exam-Teich}:
\begin{itemize}
\item[(1)] $S=L$ and $T=L$ form the initial primitive formally dual pair in $\Z_4^2$.
\item[(2)] $S_0=E_{1,1}$ and $S_1=E_{1,0}$ form a partition of $S$.
\end{itemize}
By extending this pattern, we obtain the following construction.

\begin{theorem}\label{thm-dircon2}
Let $S=T=\prod_{j=1}^m L$. Define
\begin{equation*}
S_0=\begin{cases}
  \sum_{i=0}^{\frac{m-1}{2}} E_{m,2i+1} & \mbox{if $m$ is odd,} \\
  \sum_{i=0}^{\frac{m}{2}} E_{m,2i} & \mbox{if $m$ is even,}
  \end{cases}
\end{equation*}
and
\begin{equation*}
S_1=\begin{cases}
  \sum_{i=0}^{\frac{m-1}{2}} E_{m,2i} & \mbox{if $m$ is odd,} \\
  \sum_{i=0}^{\frac{m}{2}-1} E_{m,2i+1} & \mbox{if $m$ is even,}
  \end{cases}
\end{equation*}
which form a partition of $S$. Let
\begin{equation}
\begin{aligned}\label{eqn-Tpdircon2}
S^{\pr}&=\{(0,x) \mid x \in S_0\} \cup \{(1,x) \mid x \in S_1\}, \\
T^{\pr}&=\{(0,x) \mid x \in T\} \cup \{(1,x) \mid x \in T^{(-1)}\}.
\end{aligned}
\end{equation}
Then $S^{\pr}$ and $T^{\pr}$ form a primitive formally dual pair in $\Z_2 \times \Z_4^{2m}$. Moreover, we have
$$
\theta(T^\pr,\Z_2 \times \Z_4^{2m})=2^{4m+1}-13^m-10^m.
$$
\end{theorem}

\begin{remark}\label{rem-dircon2}
\quad
\begin{itemize}
\item[(1)] The construction of Theorem~\ref{thm-dircon2} fits into the lifting construction framework \eqref{eqn-lifting2}. Note that in Theorem~\ref{thm-dircon2}, each of $S$, $S_0$ and $S_1$ is a union of basic blocks of the form $\prod_{j=1}^m N_j$, where $N_j=L_1$ or $N_j=L_2$ for each $1 \le j \le m$. We can see that $S$ is partitioned into $S_0$ and $S_1$, depending on the number of $L_1$ and $L_2$ contained in the basic blocks.
\item[(2)] Recall that $\Z_4^n$ is the additive group of the \emph{Galois ring} $\GR(4,n)$ (see \cite[Chapter 14]{Wan} for a detailed treatment of Galois rings). We remark that $L$ is the Teichmuller set in the Galois ring $\GR(4,2)$, which is used in Theorem~\ref{thm-dircon2}. It is natural to ask if we can use Teichmuller sets in the Galois rings $\GR(4,n)$, with $n \ne 2$, to construct new primitive formally dual pairs. If $n=1$, noting that $J$ in Theorem~\ref{thm-dircon1} is the Teichmuller set in $\GR(4,1)$, we refer to Theorem~\ref{thm-dircon1} and Remark~\ref{rem-dircon1}(2). When $n>2$ is odd, in Theorem~\ref{thm-dircon2}, we set $L$ to be the Teichmuller set in $\GR(4,n)$ and $m=1$. Then, the resulting pair $S^\pr$ and $T^\pr$ cannot be a primitive formally dual pair. Indeed, suppose $S$ and $T$ are the Teichmuller sets of $\GR(4,n)$. Let $S^\pr$ be an arbitrary subset of size $4^n$ and $T^{\pr}$ be the same as \eqref{eqn-Tpdircon2}. Let $z=(1,1,\ldots,1) \in \Z_4^n$. Then, $|\phi_{0,z}(T^{\pr})|^2=2^{n+1}$. By \eqref{eqn-def2}, we have $\nu_{S^{\pr}}((0,z))=\frac{1}{2}$, which is impossible. Finally, when $n=4$, in Theorem~\ref{thm-dircon2}, we set $L$ to be the Teichmuller set in $\GR(4,4)$ and $m=1$. Then with the assistance of computer, we can show that $S^{\pr}$ and $T^{\pr}$ do not form a primitive formally dual pair.
\end{itemize}
\end{remark}

To prove Theorem~\ref{thm-dircon2}, we need more notation. Define three subsets of $\Z_4^2$ as follows:
\begin{align*}
Z&=\{(0,2),(2,0),(2,2)\}, \\
Y&=\{(0,0),(0,2),(2,0),(2,2)\}, \\
I&=\{(0,1),(0,3),(1,0),(3,0),(1,1),(3,3)\}.
\end{align*}
Furthermore, define
\begin{align*}
M_m&=\{(z_1,z_2,\ldots,z_m) \in \Z_4^{2m} \mid \mbox{each $z_i \in \Z_4^2$ and there exists $z_j \in Z$} \}, \\
O_m&=\{(z_1,z_2,\ldots,z_m) \in \Z_4^{2m} \mid \mbox{each $z_i \in \Z_4^2$, there exsits $z_j \in \Z_4^2 \sm (Y \cup I)$} \}.
\end{align*}

We have the following preparatory lemma.

\begin{lemma}\label{lem-TTdiff2}
Let $T=\prod_{j=1}^m L$.
\begin{itemize}
\item[(1)] For $x=(x_1,x_2,\ldots,x_m) \in \Z_4^{2m}$, with $x_i \in \Z_4^2$,
$$
[TT^{(-1)}]_x=\begin{cases}
  0 & \mbox{if $x \in M_{m}$,} \\
  4^{l} & \mbox{if $x \notin M_{m}$ and $|\{1 \le i \le m \mid x_i \in \{(0,0)\}\}|=l$.}
\end{cases}
$$
In particular, $|\{TT^{(-1)}\}|=13^m$.
\item[(2)] For $x=(x_1,x_2,\ldots,x_m) \in \Z_4^{2m}$, with $x_i \in \Z_4^2$,
$$
[TT]_x=[T^{(-1)}T^{(-1)}]_x=\begin{cases}
  0 & \mbox{if $x \in O_{m}$,} \\
  2^{l} & \mbox{if $x \notin O_{m}$ and $|\{1 \le i \le m \mid x_i \in I\}|=l$.}
\end{cases}
$$
In particular, $|\{TT\}|=|\{T^{(-1)}T^{(-1)}\}|=|\{TT+T^{(-1)}T^{(-1)}\}|=10^m$.
\end{itemize}
\end{lemma}
\begin{proof}
We only prove (2), since the proof of (1) is similar. We have $TT=\prod_{i=1}^m LL$ and $T^{(-1)}T^{(-1)}=\prod_{i=1}^m L^{(-1)}L^{(-1)}$. Note that
\begin{equation}\label{eqn-eqn22}
LL=L^{(-1)}L^{(-1)}=Y+2I.
\end{equation}
Therefore, $TT=T^{(-1)}T^{(-1)}$ and $TT+T^{(-1)}T^{(-1)}=2TT$, which implies $|\{TT\}|=|\{T^{(-1)}T^{(-1)}\}|=|\{TT+T^{(-1)}T^{(-1)}\}|$. If $x \in O_m$, we have $[TT]_x=0$ . If $x \notin O_m$, then $x_i \in Y \cup I$ for each $1 \le i \le m$. By~\eqref{eqn-eqn22}, if $x_i \in Y$, then there is a unique way to express $x_i$ as a sum of elements from $L$. Similarly, if $x_i \in I$, then there are two ways to express $x_i$ as a sum of elements from $L$. Suppose $|\{1 \le i \le m \mid x_i \in I\}|=l$. Then, we have $[TT]_x=2^l$. Finally, $|\{TT\}|=4^{2m}-|O_m|=10^m$.
\end{proof}

Now we proceed to prove Theorem~\ref{thm-dircon2}.

\begin{proof}[Proof of Theorem~\ref{thm-dircon2}]
The proof is by induction. If $m=1$, then the conclusion of Theorem~\ref{thm-dircon2} follows from Example~\ref{exam-Teich}. The induction assumption is that the conclusion of Theorem~\ref{thm-dircon2} holds for $m=k$, and we are going to prove that the conclusion is true for $m=k+1$.

First, assume that $k$ is odd. Let $S_1=T_1=\prod_{j=1}^k L$. By Proposition~\ref{prop-prod}, $S_1$ and $T_1$ form a primitive formally dual pair in $\Z_4^{2k}$. Let $S_{10}$ and $S_{11}$ form a partition of $S_1$, where
$$
S_{10}=\sum_{i=0}^{\frac{k-1}{2}} E_{k,2i+1}, \quad S_{11}=\sum_{i=0}^{\frac{k-1}{2}} E_{k,2i}.
$$
Define two subsets of $\Z_2 \times \Z_4^{2k}$ as
\begin{align*}
S_1^{\pr}&=\{(0,x) \mid x \in S_{10}\} \cup \{(1,x) \mid x \in S_{11}\}, \\
T_1^{\pr}&=\{(0,x) \mid x \in T_1\} \cup \{(1,x) \mid x \in T_1^{(-1)}\}.
\end{align*}
By the induction assumption, we know that $S_1^{\pr}$ and $T_1^{\pr}$ form a primitive formally dual pair in $\Z_2 \times \Z_4^{2k}$. Let $S_2=T_2=L$, which form a primitive formally dual pair in $\Z_4^{2}$. Let $S_{20}=L_1$ and $S_{21}=L_2$ form a partition of $S_2$. Define
\begin{align*}
S_2^{\pr}&=\{(0,x) \mid x \in S_{20}\} \cup \{(1,x) \mid x \in S_{21}\}, \\
T_2^{\pr}&=\{(0,x) \mid x \in T_2\} \cup \{(1,x) \mid x \in T_2^{(-1)}\}.
\end{align*}
By Example~\ref{exam-Teich}, we know that $S_2^{\pr}$ and $T_2^{\pr}$ form a primitive formally dual pair in $\Z_2 \times \Z_4^{2}$. It is easy to verify that $\chi^2(T_2) \in \R$ for each $\chi \in \wh{\Z_4^2}$. Define
\begin{align*}
S^{\pr\pr}&=\{(0,x_1,x_2) \mid (x_1,x_2) \in S_0^{\pr\pr}\} \cup \{(1,x_1,x_2) \mid (x_1,x_2) \in S_1^{\pr\pr}\}, \\
T^{\pr\pr}&=\{(0,x_1,x_2) \mid (x_1,x_2) \in T_1 \times T_2\} \cup \{(1,x_1,x_2) \mid (x_1,x_2) \in T_1^{(-1)} \times T_2^{(-1)}\},
\end{align*}
where
\begin{align*}
S_0^{\pr\pr}&=(S_{10} \times S_{20}) \cup (S_{11} \times S_{21})=\sum_{i=0}^{\frac{k+1}{2}} E_{k+1,2i}, \\
S_1^{\pr\pr}&=(S_{10} \times S_{21}) \cup (S_{11} \times S_{20})=\sum_{i=0}^{\frac{k-1}{2}} E_{k+1,2i+1}.
\end{align*}
Then, by Theorem~\ref{thm-recur}, we know that $S^{\pr\pr}$ and $T^{\pr\pr}$ form a primitive formally dual pair in $\Z_2 \times \Z_4^{2k+2}$. Moreover, by Theorem~\ref{thm-recur} and Lemma~\ref{lem-TTdiff2}, we have
\begin{align*}
\theta(T^{\pr\pr},\Z_2 \times \Z_4^{2k+2})&=2\cdot 4^{2k+2}-|\{T_1T_1^{(-1)}\}|\cdot|\{T_2T_2^{(-1)}\}|-|\{T_1T_1+T_1^{(-1)}T_1^{(-1)}\}|\cdot |\{T_2T_2\}| \\
&=2\cdot4^{2k+2}-13^k\cdot13-10^k\cdot10\\
&=2^{4k+5}-13^{k+1}-10^{k+1}.
\end{align*}
Hence, we have shown that the conclusion of Theorem~\ref{thm-dircon2} is true for $m=k+1$ when $k$ is odd.

Second, assume that $k$ is even. Using a similar argument as in the $k$ odd case, we can show that the conclusion of Theorem~\ref{thm-dircon2} is true for $m=k+1$ when $k$ is even.
\end{proof}

Employing the recursive construction framework \eqref{eqn-recur}, we can derive more primitive formally dual pairs in $\Z_2 \times \Z_4^{2m}$, by combining the two constructions of Theorems~\ref{thm-dircon1} and~\ref{thm-dircon2}.

\begin{theorem}\label{thm-mix}
Let $m$ be a positive integer, satisfying $m=m_1+m_2$, where $0 \le m_1,m_2 \le m$. Let $S_1=T_1=\prod_{j=1}^{2m_1} J$. Define
$$
S_{10}=\sum_{\substack{0 \le i \le 2m_1 \\ i \equiv 0,1 \bmod 4}}B_{m_1,i}, \quad S_{11}=\sum_{\substack{0 \le i \le 2m_1 \\ i \equiv 2,3 \bmod 4}}B_{m_1,i}
$$
which form a partition of $S_1$. Let $S_2=T_2=\prod_{j=1}^{m_2} L$. Define
$$
S_{20}=\begin{cases}
  \sum_{i=0}^{\frac{m_2-1}{2}} E_{m_2,2i+1} & \mbox{if $m_2$ is odd,} \\
  \sum_{i=0}^{\frac{m_2}{2}} E_{m_2,2i} & \mbox{if $m_2$ is even.}
  \end{cases}
$$
and
$$
S_{21}=\begin{cases}
  \sum_{i=0}^{\frac{m_2-1}{2}} E_{m_2,2i} & \mbox{if $m_2$ is odd,} \\
  \sum_{i=0}^{\frac{m_2}{2}-1} E_{m_2,2i+1} & \mbox{if $m_2$ is even.}
  \end{cases}
$$
which form a partition of $S_2$. Define two subsets of $\Z_2 \times \Z_4^{2m}$ as
\begin{align*}
S^\pr&=\{(0,x_1,x_2) \mid (x_1,x_2) \in S_0^\pr\} \cup \{(1,x_1,x_2) \mid (x_1,x_2) \in S_1^\pr\}, \\
T^\pr&=\{(0,x_1,x_2) \mid (x_1,x_2) \in T_1 \times T_2\} \cup \{(1,x_1,x_2) \mid (x_1,x_2) \in T_1^{(-1)} \times T_2^{(-1)}\},
\end{align*}
where
\begin{align*}
S_0^\pr&=(S_{10} \times S_{20}) \cup (S_{11} \times S_{21}), \\
S_1^\pr&=(S_{10} \times S_{21}) \cup (S_{11} \times S_{20}).
\end{align*}
Then $S^\pr$ and $T^\pr$ form a primitive formally dual pair in $\Z_2 \times \Z_4^{2m}$. Moreover,
$$
\theta(T^{\pr},\Z_2 \times \Z_4^{2m})=2^{4m+1}-3^{2m_1}13^{m_2}-(2\cdot3^{2m_1}-2^{2m_1})10^{m_2}.
$$
\end{theorem}
\begin{proof}
Define
\begin{align*}
\wti{S_1}&=\{(0,x) \mid x \in S_{10}\} \cup \{(1,x) \mid x \in S_{11}\}, \\
\wti{T_1}&=\{(0,x) \mid x \in T_1\} \cup \{(1,x) \mid x \in T_1^{(-1)}\},
\end{align*}
and
\begin{align*}
\wti{S_2}&=\{(0,x) \mid x \in S_{20}\} \cup \{(1,x) \mid x \in S_{21}\}, \\
\wti{T_2}&=\{(0,x) \mid x \in T_2\} \cup \{(1,x) \mid x \in T_2^{(-1)}\}.
\end{align*}
By Theorem~\ref{thm-dircon1}, $\wti{S_1}$ and $\wti{T_1}$ form a primitive formally dual pair in $\Z_2 \times \Z_4^{2m_1}$. By Theorem~\ref{thm-dircon2}, $\wti{S_2}$ and $\wti{T_2}$ form a primitive formally dual pair in $\Z_2 \times \Z_4^{2m_2}$. Note that for each $\chi \in \wh{\Z_4^{2m_2}}$, we have $\chi^2(T_2)=\chi(T_2T_2)=\chi(LL)^{m_2}$. Thus, it is easy to see that $\chi^2(T_2) \in \R$ for each $\chi \in \wh{\Z_4^{2m_2}}$. By Theorem~\ref{thm-recur}, we conclude that $S^\pr$ and $T^\pr$ form a primitive formally dual pair in $\Z_2 \times \Z_4^{2m}$. By \eqref{eqn-eqn3} and \eqref{eqn-eqn6}, we have
$$
|\{T_1T_1^{(-1)}\}|=3^{2m_1}, \quad |\{T_1T_1+T_1^{(-1)}T_1^{(-1)}\}|=2\cdot3^{2m_1}-2^{2m_1}.
$$
Furthermore, together with Theorem~\ref{thm-recur} and Lemma~\ref{lem-TTdiff2}, we have
\begin{align*}
\theta(T^{\pr},\Z_2 \times \Z_4^{2m})&=2\cdot 4^{2m}-|\{T_1T_1^{(-1)}\}|\cdot|\{T_2T_2^{(-1)}\}|-|\{T_1T_1+T_1^{(-1)}T_1^{(-1)}\}|\cdot |\{T_2T_2\}|\\
&=2^{4m+1}-3^{2m_1}13^{m_2}-(2\cdot3^{2m_1}-2^{2m_1})10^{m_2}.
\end{align*}
\end{proof}

\begin{remark}
In Theorem~\ref{thm-mix}, we allow that $m_1=0$ or $m_2=0$. Indeed, if $m_1=0$, we reproduce the primitive formally dual pairs in Theorem~\ref{thm-dircon2}, and if $m_2=0$, we reproduce the primitive formally dual pairs in Theorem~\ref{thm-dircon1}.
\end{remark}

Finally, we note that Theorem~\ref{thm-mix} leads to a series of inequivalent primitive formally dual pairs in $\Z_2 \times \Z_4^{2m}$.

\begin{theorem}
For $m \ge 2$, there exist at least $m+1$ pairwise inequivalent primitive formally dual pairs in $\Z_2 \times \Z_4^{2m}$.
\end{theorem}
\begin{proof}
Given a positive integer $m \ge 2$, there are $m+1$ different ways to write $m=m_1+m_2$, where $0 \le m_1,m_2 \le m$. Applying Theorem~\ref{thm-mix}, we obtain $m+1$ distinct primitive formally dual pairs in $\Z_2 \times \Z_4^{2m}$, where
$$
\theta(T^{\pr},\Z_2\times \Z_4^{2m})=2^{4m+1}-\left(\frac{9}{13}\right)^{m_1}13^m-\left(2\left(\frac{9}{10}\right)^{m_1}-\left(\frac{2}{5}\right)^{m_1}\right)10^m.
$$
When $m=2$, we have $m_1 \in \{0,1,2\}$ and correspondingly $\theta(T^{\pr},\Z_2\times \Z_4^{4}) \in \{243,255,285\}$. When $m \ge 3$, it is easy to verify that the function
$$
f_m(x)=2^{4m+1}-\left(\frac{9}{13}\right)^{x}13^m-\left(2\left(\frac{9}{10}\right)^{x}-\left(\frac{2}{5}\right)^{x}\right)10^m
$$
is strictly increasing in the interval $x \in [0,m]$, by considering the derivative of $f_m(x)$. Thus, when $m \ge 2$, for different choices of $m_1$, the numbers $\theta(T^{\pr},\Z_2\times \Z_4^{2m})$ are distinct, which implies that the subsets $T^\pr$ have distinct difference spectra and therefore are pairwise inequivalent. Moreover, by \eqref{eqn-def}, the number $\theta(T^{\pr},\Z_2\times \Z_4^{2m})$ is equal to the frequency of $0$ in the character spectrum of $S^{\pr}$. Hence, for different choices of $m_1$, the subsets $S^{\pr}$ have distinct character spectra and are pairwise inequivalent. Consequently, by Definition~\ref{def-equiv}, when $m \ge 2$, the $m+1$ primitive formally dual pairs in $\Z_2 \times \Z_4^{2m}$ are pairwise inequivalent. When $m=1$, applying Theorem~\ref{thm-mix} reproduces Examples~\ref{exam-TITO} and \ref{exam-Teich}, which are equivalent with each other.
\end{proof}


%


\section{Concluding Remarks}\label{sec7}

In this paper, we proposed a lifting construction framework and a recursive construction framework of primitive formally dual pairs. Applying the lifting construction framework, we obtained the first infinite family of primitive formally dual pairs in $\Z_2 \times \Z_4^{2m}$, having subsets with unequal sizes. Applying the recursive construction framework, we derived the second infinite family in $\Z_2 \times \Z_4^{2m}$ with the same property. Moreover, by combining these two families, the recursive construction framework generated more primitive formally dual pairs in $\Z_2 \times \Z_4^{2m}$. As a consequence, we showed that for $m \ge 2$, there exist at least $m+1$ pairwise inequivalent primitive formally dual pairs in $\Z_2 \times \Z_4^{2m}$. All primitive formally dual pairs constructed in this paper satisfy that the two subsets have unequal sizes. Prior to our work, there was only one single example of such primitive formally dual pair.

The formally dual pair indicates how one can form periodic configurations by taking the union of translations of a given lattice. In this sense, our new constructions of formally dual pairs lead to schemes generating candidates of energy-minimizing periodic configurations.

We think the approach proposed in this paper deserves further investigation. Below, we mention several natural problems which seem to be interesting.
\begin{itemize}
\item[(1)] Note that $J$ and $L$ are the Teichmuller sets of Galois rings $\GR(4,1)$ and $\GR(4,2)$, which are the fundamental building blocks of the constructions in Theorem~\ref{thm-dircon1} and Theorem~\ref{thm-dircon2}, respectively. A natural idea is to consider whether the Teichmuller set of a general Galois ring $\GR(p^t,n)$ can be used to construct new primitive formally dual pairs. In this direction, a series of fruitful insights into Galois rings \cite{CRX,HLX,Pol,Ya,YY} may be helpful.

\item[(2)] We remark that a Teichmuller set of the Galois ring $\GR(4,n)$ forms a relative difference set in the additive group $\Z_4^n$ of $\GR(4,n)$ (see \cite[Section 2]{Pott95} for an introduction to relative difference sets). We ask if relative difference sets other than those derived from Teichmuller sets can be used to generate new primitive formally dual pairs. In this sense, the constructions in Theorems~\ref{thm-dircon1}, ~\ref{thm-dircon2} and \ref{thm-mix} might just be part of a bigger picture.

\item[(3)] We think the proposed lifting construction framework and recursive construction framework are of great interest. It is worthwhile to consider whether these frameworks can be used to generate primitive formally dual pairs in finite abelian groups other than $\Z_2 \times \Z_4^{2m}$. We note that the lifting construction framework resembles the so called Waterloo decomposition of Singer difference sets \cite{ADJP,ADLM}. Moreover, it would be very nice if one can find a way to appreciate the recursive construction framework from the viewpoint of the recursive approach based on building sets \cite{DJ}.

\item[(4)] In \cite[Table A.1]{LPS}, all primitive formally dual sets in finite abelian groups of order at most $63$ were classified. The smallest open cases of formally dual pairs having subsets with unequal sizes live in finite abelian groups $G$ of order $64$, with $G \in \{ \Z_4 \times \Z_{16}, \Z_2^2 \times \Z_{16}, \Z_8^2, \Z_2 \times \Z_4 \times \Z_8, \Z_4^3 \}$, where the two subsets have size $4$ and $16$, respectively. One may expect that an exponent bound on the group containing primitive formal dual pairs will rule out the group like $\Z_2^2 \times \Z_{16}$. In this regard, some deep number-theoretic approach \cite{LS,S} may help.
\end{itemize}

\section*{Acknowledgement}

Shuxing Li is supported by the Alexander von Humboldt Foundation.

\end{document}